%% file: local-param.tex
\pdfoutput=1
\documentclass{elsarticle}
\usepackage{mystyle}

\journal{Applied and Computational Harmonic Analysis}

\begin{document}
\sloppy

\begin{frontmatter}

\title{Local Behavior of Sparse Analysis Regularization: Applications to Risk Estimation}

\author{Samuel~Vaiter}
\ead{samuel.vaiter@ceremade.dauphine.fr}
\author{Charles-Alban~Deledalle}
\ead{deledalle@ceremade.dauphine.fr}
\author{Gabriel~Peyr\'{e}}
\ead{gabriel.peyre@ceremade.dauphine.fr}
\address{CEREMADE, CNRS, Universit\'{e} Paris-Dauphine, Place du Mar\'{e}chal De Lattre De Tassigny, 75775 Paris Cedex 16, France}

\author{Charles~Dossal}
\ead{charles.dossal@math.u-bordeaux1.fr}
\address{IMB, Universit\'{e} Bordeaux 1, 351, Cours de la Lib\'{e}ration, 33405 Talence Cedex, France}

\author{Jalal~Fadili}
\ead{Jalal.Fadili@greyc.ensicaen.fr}
\address{GREYC, CNRS-ENSICAEN-Universit\'{e} de Caen, 6, Bd du Mar\'{e}chal Juin, 14050 Caen Cedex, France}

\begin{abstract}
\input{sections/abstract.tex}
\end{abstract}


\end{frontmatter}


\input{sections/intro.tex}
\input{sections/contrib.tex}

\input{sections/pw.tex}
\input{sections/examples.tex}
\input{sections/conclusion.tex}

\appendix
\input{sections/proofs.tex}

\bibliographystyle{model1-num-names}
\bibliography{local-param}

\end{document}

%% file: sections/abstract.tex

In this paper, we aim at recovering an unknown signal $x_0$ from noisy measurements $y = \Phi x_0 + w$, where $\Phi$ is an ill-conditioned or singular linear operator and $w$ accounts for some noise. To regularize such an ill-posed inverse problem, we impose an analysis sparsity prior. More precisely, the recovery is cast as a convex optimization program where the objective is the sum of a quadratic data fidelity term and a regularization term formed of the $\lun$-norm of the correlations between the sought after signal and atoms in a given (generally overcomplete) dictionary. The $\lun$-sparsity analysis prior is weighted by a regularization parameter $\lambda > 0$. In this paper, we prove that any minimizers of this problem is a piecewise-affine function of the observations $y$ and the regularization parameter $\lambda$. As a byproduct, we exploit these properties to get an objectively guided choice of $\lambda$. In particular, we develop an extension of the Generalized Stein Unbiased Risk Estimator (GSURE) and show that it is an unbiased and reliable estimator of an appropriately defined risk. The latter encompasses special cases such as the prediction risk, the projection risk and the estimation risk. We apply these risk estimators to the special case of $\lun$-sparsity analysis regularization. We also discuss implementation issues and propose fast algorithms to solve the $\lun$ analysis minimization problem and to compute the associated GSURE. We finally illustrate the applicability of our framework to parameter(s) selection on several imaging problems.

%% file: sections/intro.tex
\section{Introduction} 
\label{sec:intro}

\subsection{Regularization of Linear Inverse Problems} 
\label{sub:intro-linear}

In many applications, the goal is to recover an unknown signal $x_0 \in \RR^N$ from noisy and linearly degraded observations $y \in \RR^Q$.
The forward observation model reads
\begin{equation}\label{eq:linear-problem}
	y = \Phi x_0 + w,
\end{equation}
where $w \in \RR^Q$ is the noise component and the mapping $\Phi: \RR^N \rightarrow \RR^Q$ is a known linear operator which generally models an acquisition process that entails loss of information so that $Q \leq N$. Even when $Q=N$, $\Phi$ is typically ill-conditioned or even rank-deficient.
In image processing, typical applications covered by the above degradation model are entry-wise masking (inpainting), convolution (acquisition blur), Radon transform (tomography) or a random sensing matrix (compressed sensing).

Solving for an accurate approximation of $x_0$ from the system \eqref{eq:linear-problem} is generally ill-posed~\cite{kirsch2011introduction}.
In order to regularize them and reduce the space of candidate solutions, one has to incorporate some prior knowledge on the typical structure of the original object $x_0$. This prior information accounts for the smoothness of the solution and can range from uniform smoothness assumption to more complex geometrical priors.

Regularization is a popular way to impose such a prior, hence making the search for solutions feasible. The general variational problem we consider can be stated as
\begin{equation}\label{eq:reg-noise}
  \solS \in \uArgmin{x \in \RR^N} F(x,y) + \lambda R(x) ,
\end{equation}
where $F$ is the so-called data fidelity term, $R$ is an appropriate regularization term that encodes the prior on the sought after signal, and $\lambda > 0$ a regularization parameter. This parameter balances between the amount of allowed noise and the regularity dictated by $R$.
In this paper, we consider a quadratic data fidelity term taking the form
\begin{equation}
	F(x,y) = \dfrac{1}{2} \normd{y - \Phi x}^2 ~.
\end{equation}
If it were to be interpreted in Bayesian language, this data fidelity would amount to assuming that the noise is white Gaussian.

The popular Thikonov class of regularizations corresponds to quadratic forms $R(x) = \dotp{x}{Kx}$, where $K$ is a symmetric semidefinite positive kernel. This typically induces some kind of uniform smoothness on the recovered vector. To capture the more intricate geometrical complexity of image structures, non-quadratic priors are required, among which sparse regularization through the $\lun$ norm has received a considerable interest in the recent years. This non-smooth regularization is at the heart of this paper.

\subsection{Sparse $\lun$-Analysis Regularization} 
\label{sub:intro-sparse}

We call a dictionary $D = (d_i)_{i=1}^P$ a collection of $P$ atoms $d_i \in \RR^N$. The dictionary may be redundant in $\RR^N$, in which case $P > N$ and $\Phi$ if surjective if it has full row rank. $D$ can also be viewed as a linear mapping from $\RR^P$ to $\RR^N$ which is used to \emph{synthesize} a signal $x \in \Span(D) \subseteq \RR^N$ as $x = D \alpha = \sum_{i=1}^P \alpha_i d_i$, where $\alpha$ is not uniquely defined if $D$ is a redundant dictionary.

\medskip

The $\lun$ analysis regularization in a dictionary $D$ corresponds to using $R = R_A$ in \eqref{eq:reg-noise} where
\begin{equation}
	R_A(x) = \normu{D^* x} .
\end{equation}
This leads us to the following minimization problem which is the focus of this paper
\begin{equation}\label{eq:lasso-a}\tag{\lassotag}
  \solS \in \uArgmin{x \in \RR^N} \minform .
\end{equation}
Since the objective function in \lasso is proper (i.e. not infinite everywhere), continuous and convex, the set of (global) minimizers of \lasso is nonempty and compact if, and only if,
\begin{equation}\label{eq:H0}\tag{$H_0$}
  \Ker \Phi \cap \Ker D^*= \ens{0} ,
\end{equation}
All throughout this paper, we suppose that this condition holds.

\medskip 

The most popular $\lun$-analysis sparsity-promoting regularization is the total variation, which was first introduced for denoising (in a continuous setting) in~\cite{rudin1992nonlinear}.
In a discrete setting, it corresponds to taking $D^*$ as a finite difference discretization of the gradient operator.
The corresponding prior $\J_A$ favors piecewise constant signals and images.
A comprehensive review of total variation regularization can be found in~\cite{chambolle2009tv}.

The theoretical properties of total variation regularization have been previously investigated.
A distinctive feature of this regularization is its tendency to yield a staircasing effect, where discontinuities not present in the original data might be artificially created by the regularization.
This effect has been studied by Nikolova in the discrete case in a series of papers, see e.g. \cite{nikolova2000local}, and by Ring in \cite{Ring2000} in the continuous setting. The stability of the discontinuity set of the solution of the 2-D continuous total variation denoising is studied in~\cite{caselles2008discontinuity}.

\medskip 

When $D$ is the standard basis, i.e. $D = \Id$, the analysis sparsity regularization $R_A$ specializes to the so-called synthesis regularization. The corresponding variational problem \lasso is referred to as the Lasso problem in the statistics community~\cite{tibshirani1996regre} and Basis-Pursuit DeNoising (BPDN) in the signal processing community~\cite{chen1999atomi}. Despite synthesis and analysis regularizations both minimize objective functions involving the $\lun$ norm, the properties of their respective minimizers may differ significantly. Some insights on the relation and distinction between analysis and synthesis-based sparsity regularizations were first given in \cite{elad2007analy}. When $D$ is orthogonal, and more generally when $D$ is square and invertible, analysis and synthesis regularizations are equivalent in the sense that the set of minimizers of one problem can be retrieved from that of an equivalent form of the other through a bijective change of variable. However, when $D$ is redundant, synthesis and analysis regularizations depart significantly. 

While the theoretical guarantees of synthesis $\lun$-regularization have been extensively studied, the analysis case remains much less investigated ~\cite{candes2010compr,grasmair2011linear,nam2011the-c,vaiter-analysis}.

\subsection{Geometrical Insights into $\lun$-Analysis Regularization} 
\label{sub:intro-geometry}

In the synthesis prior, sparsity of a vector $\alpha \in \RR^P$ is measured in terms of its $\lzero$ pseudo-norm, or equivalently the cardinality of its support $\supp(\alpha)$, i.e.
\begin{equation*}
  \norm{\alpha}_0 = \abs{\supp(\alpha)} = \abs{\enscond{i \in \ens{1, \cdots, P}}{\alpha_i \neq 0}} .
\end{equation*}

In the analysis prior, the sparsity is measured on the correlation vector $D^* x$. It then appears natural to keep track of the support of $D^* x$. To fix terminology, we define this support and its complement.
\begin{defn}
  The \emph{\dsup} $I$ (resp. \emph{\dcosup} $J$) of a vector $x \in \RR^N$ is defined as $I =\supp(D^* x)$ (resp. $J = I^c=\{1,\dots,P\} \setminus I$).
\end{defn}

A vector $x$ with a \dcosup $J$ is then such that the correlations between this vector and the columns of $D_J$ are zero. This is equivalent to saying that $x$ lives in a subspace $\GJ$ defined as follows.
\begin{defn}\label{defn:cospace}
  Given $J$ a subset of $\ens{1\cdots P}$, the \emph{cospace} $\GJ$ is defined as
  \begin{equation*}
    \GJ = \Ker D_J^* .
  \end{equation*}
\end{defn}
It was shown in \cite{vaiter-analysis} that the subspace $\GJ$ plays a pivotal role in robustness and identifiability guarantees of \lasso.

\medskip

In fact, the subspaces $\GJ$ carry all necessary information to characterize signal models of sparse analysis type. More precisely, vectors of cosparsity $k=\abs{J}$ live in an \emph{union of subspaces}
\begin{equation*}
  \Theta_k = \enscond{\GJ}{J \subseteq \ens{1\cdots P} \qandq \dim \GJ = k} ~,
\end{equation*}
and the signal space $\RR^N$ can be decomposed as $\RR^N = \bigcup_{k \in \ens{0, \dots, N}} \Theta_k$. This model has been introduced in \cite{nam2011the-c} under the name \emph{cosparse model}.

For synthesis sparsity, i.e. $D=\Id$, $\Theta_k$ are nothing but the set of axis-aligned subspaces of dimension $k$.
For the 1-D total variation prior, where $D$ corresponds to finite forward differences, $\Theta_k$ is the set of piecewise constant signals with $k-1$ steps.
A few other examples of subspaces $\Theta_k$, including those corresponding to translation invariant wavelets, are discussed in~\cite{nam2011the-c}. More general union of subspaces models have been introduced in sampling theory to model various types of non-linear signal ensembles, see e.g.~\cite{lu2008a-the}.


\subsection{Local Behavior of Minimizers} 
\label{sub:intro-local}

Local variations and sensitivity/perturbation analysis of problems in the form of \eqref{eq:reg-noise} is an important topic in optimization and optimal control. Comprehensive monograph on the subject are~\cite{bonnans2000perturbation,mordukhovich1992sensitivity}.
In this paper, we focus on the variations with respect to the regularization parameter $\lambda$ and the observations $y$, i.e we study the set-valued mapping $(\lambda,y) \mapsto \Mm_\lambda(y)$ where $\Mm_\lambda(y)$ is the set of minimizers of \eqref{eq:reg-noise}.

In the synthesis case ($D=\Id$) with $Q > N$, the work of~\cite{osborne2000lasso,osborne2000new} showed that, for a fixed $y$, the mapping $\lambda \mapsto \solS$ is piecewise affine, i.e. the solution path is polygonal. Further, they characterized changes in the solution $\solS$ at vertices of this path. Based on these observations, they presented the homotopy algorithm, which follows the solution path by jumping from vertex to vertex of this polygonal path. This idea was extended to the underdetermined case in~\cite{malioutov2005homotopy,donoho2008fast}. A homotopy-type scheme was proposed in \cite{tibshirani2011solution} for sparse $\lun$-analysis regularization in the overdetermined case ($Q > N$). We will discuss the latter work in more detail
in Section~\ref{sec:related}.


\subsection{Risk Estimation and Parameter Selection} 
\label{sub:risk}

This paper also addresses unbiased estimation of the $\ldeux$-risk when recovering a vector $x_0 \in \RR^N$ from the measurements $y$ in \eqref{eq:linear-problem}, e.g. by solving \eqref{eq:reg-noise}, under the assumption that $w$ is white Gaussian noise.
A central concept for risk estimation is that of the
degrees of freedom (DOF). 
Let $\solTY$ be an estimator of $x_0$ from \eqref{eq:linear-problem}, parameterized by some parameters $\theta$. 
The DOF of such an estimator was defined by Efron~\cite{efron1986biased} as
\begin{equation*}
  \DOF_{\theta}= \sum_{i=1}^Q \frac{\cov_w(y_i, (\Phi \solTY)_i)}{\sigma^2} ~.
\end{equation*}
The DOF is generally used to quantify the complexity of a statistical modeling procedure.
It plays a central role in many model validation and selection criteria, e.g.~%
Mallows' $C_{p}$ (Mallows \cite{mallows}),
AIC (Akaike information criterion~\cite{akaike1973information}),
BIC (Bayesian information criterion~\cite{schwarz1978estimating}),
GCV (generalized cross-validation~\cite{golub1979generalized}) or
SURE (Stein Unbiased Risk Estimator~\cite{stein1981estimation}).
In the spirit of the SURE theory, a good unbiased estimator of the DOF is sufficient to provide an unbiased estimate
of the $\ldeux$ risk in reconstructing $\Phi x_0$,
i.e.~the \emph{prediction} risk
$\EE_w(\normd{\Phi \solTY - \Phi x_0}^2)$.
For instance, the SURE is given by
\begin{align}
  \label{eq:sure}
  \SURE(\solTY) \!=\!
  \normd{y - \Phi \solTY}^2
  \!-\! Q \sigma^2
  \!+\! 2 \sigma^2 \widehat{\DOF}_\theta(y)
\end{align}
with
\begin{equation*}
  \widehat{\DOF}_\theta(y) \!=\! \tr \pa{ \!\frac{\partial \Phi \solTY}{\partial y} \!} ~,
\end{equation*}
where $\EE_w(\widehat{\DOF}_\theta(y)) = \DOF_{\theta}$, $\!\frac{\partial \Phi \solTY}{\partial y}$ is the Jacobian matrix of the vector function $y \mapsto \Phi \solTY$ and $\tr$ is the trace operator. \\

The SURE depends solely on $y$, without prior knowledge of $x_0$.
This can prove very useful as a basis
to automatically choose the parameters $\theta$ of the estimator,
e.g.~$\theta=\lambda$ when solving \eqref{eq:reg-noise}, or the smoothing
parameters in families of linear estimates~\cite{li1985sure}
such as for ridge regression or smoothing splines.
In some settings, it has been shown that it offers better accuracy
than GCV and related non-parametric selection techniques
\cite{efron2004estimation}. However, compared to GCV, the
SURE requires the knowledge of the noise variance $\sigma^2$.

The SURE has been widely used in the statistics and signal processing communities
as a principled and efficient way for parameter selection with a variety of non-linear estimators.
For instance, it was used for wavelet denoising \cite{donoho1995adapting,cai2009data,blu2007surelet},
wavelet shrinkage for a class of linear inverse problems \cite{johnstoneinversesure99}
and non-local filtering \cite{vandeville2009sure,duval2011abv,dds2011nlmsap}.

For general linear inverse problems, the estimator of the {prediction} risk and the parameter(s) minimizing it can depart substantially from those corresponding to the \emph{estimation} risk $\EE_w(\normd{\solTY-x_0}^2)$ \cite{rice1986choice}.
To circumvent this difficulty, in \cite{vonesch2008sure}, the authors attempted
to approximate the {estimation} risk by relying on
a regularized version of the inverse of $\Phi$.
However, in general, either $\Phi$ should have a trivial kernel or, otherwise, 
$x_0$ should live outside to $\ker(\Phi)$
to guarantee the existence of an unbiased estimator
of the {estimation} risk~\cite{desbat1995theminimum}.

In \cite{eldar-gsure}, a generalized $\SURE$ ($\GSURE$) has been
developed for noise models within the multivariate canonical exponential
family. This allows one to derive an unbiased estimator of the risk on a projected version of $\solTY$, which in turn covers the case where $\Phi$ has a non-trivial kernel and a part of $x_0$ is in it. Indeed, in the latter scenario, the $\GSURE$ can at best estimate the \emph{projection} risk $\EE_w(\normd{\Pi \solTY-\Pi x_0)}^2)$ where $\Pi$ is the orthogonal projector on $\ker(\Phi)^\perp$.



\subsection{Contributions} 
\label{sub:intro-contribution}

This paper describes the following contributions:
\begin{enumerate}[label=(\roman*)]
  \item \textbf{Local affine parameterization:}
  we show that any solution $\solS$ of \lasso is a piecewise affine function of $(y,\lambda)$. Furthermore, for fixed $\lambda$, and for $y$ outside a set of Lebesgue measure zero, the prediction $\mus$ locally varies along a constant subspace. This is a distinctly novel contribution which generalizes previously known results (see Section~\ref{sub:local_variations} for a detailed discussion). It also forms the cornerstone of unbiased estimation of the DOF. 

  \item \textbf{GSURE:}
  we derive a unifying framework to compute unbiased estimates of several risks in $\ldeux$ sense, for estimators of $x_0$ from $y$ as observed in \eqref{eq:linear-problem} when $w$ is a white Gaussian noise. This framework encompasses for instance the prediction, the projection and the estimation risks (see Section~\ref{sub:gsure} for a discussion to related work).

  \item \textbf{$\lun$-Analysis Unbiased Risk Estimation:}
  combining the results from the previous two contributions, we derive a closed-form expression of an unbiased estimator of the DOF for \lasso, whence we deduce GSURE estimates of the different risks. 

  \item \textbf{Numerical Computation of GSURE:}
  we also address in detail numerical issues that rise when implementing our DOF estimator and GSURE for \lasso. We show that the additional computational effort to compute the DOF estimator (hence the GSURE) from its closed-form is invested in solving simple linear systems. This turns out to be much faster than iterative approaches existing in the literature which are computationally demanding (see Section~\ref{sub:numgsure} for a detailed discussion).

\end{enumerate}

\subsection{Organization of the Paper} 
\label{sub:intro-organization}

The rest of the paper is organized as follows.
Section \ref{sec:contrib-local} and \ref{sub:contrib-gsure} describe each of our main contributions.
Section \ref{sec:related} draws some connections with relevant previous works.
Section \ref{sec:examples} illustrates our results on some numerical examples.
The proofs are deferred to Appendix~\ref{sec:proofs} awaiting inspection by the interested reader.

\subsection{Notation}
We first summarize the main notations used throughout the paper.
We focus on real vector spaces.
The sign vector $\sign(\alpha)$ of $\alpha \in \RR^P$ is
\begin{equation*}
  \forall i \in \ens{1, \dots, P}, \quad
  \sign(\alpha)_i =
  \begin{cases}
    + 1 & \qifq \alpha_i > 0,\\
    0 & \qifq \alpha_i = 0, \\
    - 1 & \qifq \alpha_i < 0.
  \end{cases}
\end{equation*}
Its support is
\begin{equation*}
  \supp(\alpha) = \enscond{i \in \ens{1, \dots, P}}{\alpha_i \neq 0} .
\end{equation*}
For a subset $I \subset E$, $\abs{I}$ will denote its cardinality, and $I^c=E \setminus I$ its complement.

The matrix $M_J$ for $J$ a subset of $\ens{1, \dots, P}$ is the submatrix whose columns are indexed by $J$.
Similarly, the vector $s_J$ is the restriction of $s$ to the entries of $s$ indexed by $J$.

$\tr$ and $\diverg$ are respectively the trace and divergence operators. 
The matrix $\Id$ is the identity matrix, where the underlying space will be clear from the context.
For any matrix $M$, $M^+$ is its Moore--Penrose pseudoinverse and $M^*$ is its adjoint.




%% file: sections/contrib.tex

\section{Perturbation Theory of $\lun$-Analysis Regularization} 
\label{sec:contrib-local}
Throughout this section, it is important to point out that we only require that the noise vector $w \in \RR^Q$ to be bounded. The fact that it could be deterministic or random is irrelevant here.  

\subsection{Local Affine Parameterization} 
\label{sub:contrib-affine}
Our first contribution derives a local affine parameterization of minimizers of \lasso as functions of $(y,\lambda) \in \RR^Q \times \RR_+$.
To develop our theory, the invertibility of $\Phi$ on $\GJ$ will play a vital role. For this, we need to assume that
\begin{equation}\label{eq:hj}\tag{$H_J$}
  \Ker \Phi \cap \GJ = \ens{0}.
\end{equation}
To intuitively understand the importance of this assumption, think of the ideal case where one wants to estimate a $D$-sparse signal $x_0$ from $y=\Phi x_0+w$, whose \dcosup $J$ is assumed to be known. This can be achieved by solving a least-squares problem. The latter has a unique solution if \eqref{eq:hj} holds. 

Of course, $J$ is not known in general, and one may legitimately ask whether \eqref{eq:hj} is fulfilled for some solution of \lasso. We will provide an affirmative answer to this question in Theorem~\ref{thm:dof}(ii), i.e. there always exists a solution of \lasso such that \eqref{eq:hj} holds.

With assumption \eqref{eq:hj} at hand, we now define the following matrix whose role will be clarified shortly.
\begin{defn}
  Let $J$ be a \dcosup.
  Suppose that \eqref{eq:hj} holds.
  We define the matrix $\AJ$ as
  \begin{equation}\label{eq:aj}
    \AJ = \UJ \pa{\UJ^* \Phi^* \Phi \UJ}^{-1} \UJ^* .
  \end{equation}
  where $\UJ$ is a matrix whose columns form a basis of $\GJ$.
\end{defn}
Observe that the action of $\AJ$ could be rewritten as an optimization problem
\begin{equation*}
  \AJ u = \uargmin{D_J^* x = 0} \frac{1}{2}\norm{\Phi x}^2 - \dotp{x}{u} .
\end{equation*}

\medskip

Let us now turn to sensitivity of the minimizers $\solS$ of \lasso to perturbations of $(y,\lambda)$. More precisely, our aim is to study properties, including continuity and differentiability, of $\solS$ and $\Phi\solS$ as functions of $y$ and $\lambda$. Toward this end, we will exploit the fact that $\solS$ obeys an implicit equation given in Lemma~\ref{lem:sol} (see Appendix~\ref{sub:local}). But as optimal solutions turns out to be not everywhere differentiable (change of the \dsup and thus of the cospace), we will concentrate on a local analysis where $(y,\lambda)$ vary in a small neighborhood that typically avoids non-differentiability to occur. This is exactly the reason why we introduce the transition space $\Hh$ defined below. It corresponds to the set of observation vectors $y$ and regularization parameters $\lambda$ where the cospace $\GJ$ of any solution of \lasso is not stable with respect to small perturbations of $(y,\lambda)$.
\begin{defn}\label{defn:h}
  The \emph{transition space} $\Hh$ is defined as
  \begin{equation*}
    \Hh = \bigcup_{\substack{J \subset \ens{1,\cdots,P} \\ \text{\eqref{eq:hj} holds} }}
          \bigcup_{\substack{K \subset J \\ \Im \PIJ \not\subseteq \Im D_{J \setminus K}}}
          \bigcup_{s_{J^c} \in \ens{-1,1}^{\abs{J^c}}}
          \bigcup_{\sigma_K \in \ens{-1,1}^{\abs{K}}}
          \Hh_{J,K,s_{J^c},\sigma_K},
  \end{equation*}
  where
  \begin{equation*}
    \Hh_{J,K,s_{J^c},\sigma_K} = \enscond{(y,\lambda) \in \RR^Q \times \RR_+}{\Proj_{\Gg_{J \setminus K}}\PIJ y = \Proj_{\Gg_{J \setminus K}}\lambda(\CJJ s_{J^c} - D_K \sigma_K)} ,
  \end{equation*}
  with $\PIJ = \Phi^* (\Phi \AJ \Phi^* - \Id)$, $\CJJ = (\Phi^* \Phi \AJ - \Id)D_{J^c}$ and $\Proj_{\mathcal{G}_{J \setminus K}}$ is the orthogonal projector on $\mathcal{G}_{J \setminus K}$.
\end{defn}

The following theorem summarizes our first sensitivity analysis result on the optimal solutions of \lasso.
\begin{thm}\label{thm:local}
  Let $(y, \lambda) \not\in \Hh$ and let $\solS$ be a solution of $\lasso$.
  Let $I$ and $J$ be the \dsup and \dcosup of $\solS$ and $s = \sign(D^* \solS)$.
  Suppose that \eqref{eq:hj} holds.
  For any $\bar{y} \in \RR^Q$ and $\bar{\lambda} \in \RR_+$, define
  \begin{equation*}
    \solSpLY = \AJ \Phi^* \bar{y} - \bar{\lambda} \AJ D_I s_I .
  \end{equation*}
  There exists an open neighborhood $\Bb \subset \RR^Q \times \RR_+$ of $(y,\lambda)$ such that for every $(\bar y, \bar \lambda) \in \Bb$, $\solSpLY$ is a solution of $(\lassoP{\bar{y}}{\bar{\lambda}})$.
\end{thm}

An immediate consequence of this theorem is that, for a fixed $y \in \RR^Q$, if \lasso admits a unique solution $\solS$ for each $\lambda$, then $\ens{\solS: \lambda \in \RR_+}$ identifies a polygonal solution path. As we move along the solution path, the cospace is piecewise constant as a function of $\lambda$, changing only at critical values corresponding to the vertices on the polygonal path.

\subsection{Local Variations of the Prediction} 
\label{sub:contrib-variation-prediction}
We now turn to quantifying explicitly the local variations of the prediction $\mus = \Phi \solS$ with respect to the observation $y$. First, it is not difficult to see that even if \lasso admits several solutions, all of them share the same image under $\Phi$; see Lemma~\ref{lem:unique-image} for a formal proof of this assertion. This allows to denote without ambiguity $\mus$ as a single-valued mapping. Before stating our second sensitivity analysis result, we need to define the restriction to $\RR^Q$ of the transition space $\Hh$.
\begin{defn}
  Let $\lambda \in \RR_+^*$.
  The \emph{$\lambda$-restricted transition space} is
  \begin{equation*}
    \Hh_{\cdot,\lambda} = \enscond{y \in \RR^Q}{(y,\lambda) \in \Hh} .
  \end{equation*}
\end{defn}

\medskip

\begin{thm}\label{thm:dof}
  Fix $\lambda \in \RR_+^*$. Then,
  \begin{enumerate}[label=(\roman*)]
  \item The $\lambda$-restricted transition space $\Hh_{\cdot,\lambda}$ is of Lebesgue measure zero.
  \item For $y \not\in \Hh_{\cdot,\lambda}$, there exists $\solS$ a solution of \lasso with a \dcosup $J$ that obeys \eqref{eq:hj}.
  \item The mapping $y \mapsto \mus$ is of class $C^\infty$ on $\RR^Q \setminus \Hh_{\cdot,\lambda}$ (a set of full Lebesgue measure), and
  \begin{equation}\label{eq:divdof}
    \frac{\partial \mus}{\partial y} = \Phi \AJ \Phi^* ~,
  \end{equation}
  where $J$ is such that \eqref{eq:hj} holds.
  \end{enumerate}
\end{thm}

\section{Generalized Stein Unbiased Risk Estimator} 
\label{sub:contrib-gsure}

Throughout this section, for our statements to be statistically meaningful, the noise is assumed to be white Gaussian, $w \sim \Nn(0,\sigma^2\Id_Q)$ of bounded variance $\sigma^2$. 

\subsection{GSURE for an Arbitrary Estimator}
\label{sub:contrib-arbitrary}
We first consider an arbitrary estimator $\solTY$ with parameters $\theta$ such that $\muTY=\Phi \solTY$ is a single-valued mapping. We similarly write $\mu_0 =  \Phi x_0$. Of course the results described shortly will apply when the estimator is taken as any minimizer of \lasso, in which case $\theta=\lambda$. \\

We here develop an extended version of $\GSURE$ that unbiasedly estimates the risk of reconstructing $A\mu_0$ with an arbitrary matrix $A \in \RR^{M \times Q}$. This allows us to cover in a unified framework unbiased estimation of several classical risks including the {prediction} risk (with $A = \Id$), the {projection} risk when $\Phi$ is rank deficient (with $A = \Phi^*(\Phi \Phi^*)^+$), and the {estimation} risk when $\Phi$ has full rank (with $A = \Phi^+ = (\Phi^* \Phi)^{-1} \Phi^*$). A quantity that will enter into play in the risk of estimating $A\mu_0$ is the degrees of freedom defined as
\begin{equation*}
  \DOF_{\theta}^A = \sum_{i=1}^Q \frac{\cov_w((A y)_i, (A \muTY)_i)}{\sigma^2} ~.
\end{equation*}

\begin{defn}
\label{def:gsureA}
Let $A \in \RR^{M \times Q}$.
We define the Generalized Stein Unbiased Risk Estimate (GSURE) associated to $A$ as
\begin{align*}
  \GSURE^{A}(\solTY) = &
  \normd{A(y - \muTY)}^2
  - \sigma^2 \tr( A^* A )
  + 2 \sigma^2 \dev{\theta}{A}{y} ~,
\end{align*}
where
\begin{equation*}
  \dev{\theta}{A}{y} = \tr \pa{ A \JacT A^* } ~.
\end{equation*}
\end{defn}

\paragraph{Unbiasedness of the GSURE}
The next result shows that $\GSURE^{A}(\solTY)$ is an unbiased estimator of an appropriate $\ldeux$ risk, and $\dev{\theta}{A}{y}$ is an unbiased estimator  of $\DOF_{\theta}^A$
\begin{thm}\label{thm:gsure}
  Let $A \in \RR^{M \times Q}$.
  Suppose that $y \mapsto \muTY$ is weakly differentiable, so that its divergence is well-defined in the weak sense.
  If $y = \Phi x_0 + w$ with $w \sim \Nn(0,\sigma^2\Id_Q)$,
  then
  \begin{align*}
    \EE_w \GSURE^{A}(\solTY) = \EE_w \pa{ \normd{ A \mu_0 - A \muTY}^2 }
    \qandq
    \EE_w \dev{\theta}{A}{y} = \DOF_{\theta}^A ~.
  \end{align*}
\end{thm}

\begin{rem}
Theorem~\ref{thm:gsure} can be straightforwardly adapted to deal with any white Gaussian noise
with a non-singular covariance matrix $\Sigma$. It is sufficient to consider the change of variable $y \mapsto \Sigma^{-1/2} y$
and $\Phi \mapsto \Sigma^{-1/2} \Phi$. This is in the same vein as \cite{eldar-gsure}.
\end{rem}

All estimators of the form $\GSURE^B$ with $B$ such that
$B\Phi\!=\!A\Phi$
share the same expectation given by Theorem \ref{thm:gsure}.
Hence, there are several ways to estimate
the risk in reconstructing $A\mu_0$.
For the estimation of the prediction, projection and
estimation risks, we now give the corresponding expressions
and associated estimators (with subscript notations)
as direct consequences of Theorem~\ref{thm:gsure}:

\begin{itemize}[leftmargin=*]
\item $A=\Id$: in which case $\GSURE^\Id$ becomes
\begin{align*}
  \GSURE_{\Phi}(\solTY) =
  \normd{y - \muTY}^2
  - Q \sigma^2
  + 2 \sigma^2 \tr \pa{ \JacT }
\end{align*}
which provides an unbiased estimate of the prediction risk
\begin{align*}
  \risk_{\Phi}(x_0) = \EE_w\normd{\Phi\solTY - \Phi x_0}^2 ~.
\end{align*}
This coincides with the classical $\SURE$ defined in \eqref{eq:sure}.
\item $A=\Phi^*(\Phi \Phi^*)^+$: when $\Phi$ is rank deficient, $\Pi=\Phi^*(\Phi\Phi^*)^+\Phi$ is the orthogonal projector on $\ker(\Phi)^\perp=\Im(\Phi^*)$.
Denoting $x_{\ML}(y) = \Phi^*(\Phi\Phi^*)^+ y$ the maximum likelihood estimator (MLE), $\GSURE^{\Phi^*(\Phi \Phi^*)^+}$ becomes 
\begin{align*}
  \GSURE_\Pi(\solTY) \! = \! \normd{x_{\ML}(y) - \Pi \solTY}^2
  \!-\! \sigma^2 \tr \!\pa{ (\Phi \Phi^*)^{+} } \\
  \!+\! 2 \sigma^2 \tr \!\pa{ \!(\Phi\Phi^*)^+ \JacT } \!.
\end{align*}
It provides an unbiased estimate of the {projection} risk
\begin{align*}
  \risk_{\Pi}(x_0) = \EE_w\normd{\Pi\solTY - \Pi x_0}^2 ~.
\end{align*}

If $\Phi$ is the synthesis operator of a Parseval tight frame, i.e.~$\Phi \Phi^* = \Id$, the {projection} risk coincides with the {prediction} risk and so do the corresponding $\GSURE$ estimates
\begin{align*}
  \risk_{\Pi}(x_0) = \risk_{\Phi}(x_0)
  \quad \text{and} \quad
  \GSURE_{\Pi}(\solTY) = \GSURE_\Phi(\solTY) ~.
\end{align*}

It is also worth noting that if $\solTY$ never lies in $\ker(\Phi)$, then $\risk_{\Pi}(x_0)$ coincides with the estimation risk up to the additive constant $\normd{(\Id-\Pi)x_0}^2$. 

\item $A=(\Phi^* \Phi)^{-1}\Phi^*$: in this case $\Phi$ has full rank, and the mapping $y \mapsto \solTY$ is single-valued and weakly differentiable. The maximum likelihood estimator is now $x_{\ML}(y) = (\Phi^*\Phi)^{-1} \Phi^* y$ , and $\GSURE^{(\Phi^* \Phi)^{-1} \Phi^*}$ takes the form
\begin{align*}
  \GSURE_\Id(\solTY) = \normd{x_{\ML}(y) - \solTY}^2
   -  \sigma^2 \tr \pa{ (\Phi^* \Phi)^{-1} }\\
   +  2 \sigma^2 \tr\pa{ \Phi (\Phi^*\Phi)^{-1} \dfrac{\partial \solTY}{\partial y} } ~ .
\end{align*}
This is an unbiased estimator of the {estimation} risk given by
\begin{align*}
  \risk_{\Id}(x_0) = \EE_w\normd{\solTY - x_0}^2 ~.
\end{align*}
\end{itemize}

\paragraph{Reliability of the GSURE}
We now assess the reliability of the $\GSURE$ by computing
the expected squared-error between $\GSURE^A(\solTY)$ and the true
squared-error on $A \mu_0$
\begin{align*}
  \SE^A(\solTY) = \normd{ A \mu_0 - A \muTY}^2 ~.
\end{align*}

\begin{thm}\label{thm:reliability}
  Under the assumptions of Theorem~\ref{thm:gsure},
  we have
  \begin{align*}
    \EE_w&\left[ \left(\GSURE^A(\solTY) - \SE^A(\solTY)\right)^2\right]
    \!= \\
    &\hspace{1em} 2 \sigma^4 \tr\left[(A^* A)^2\right] \\
    &+
    4 \sigma^2 \EE_w \normd{A^* A (\mu_0 - \muTY)}^2 \\
    & -
    4 \sigma^4
    \EE_w \!\left( \tr\left[ A \JacT A^* A \left( 2 \Id - \JacT \right) A^* \right] \right)~.
  \end{align*}
\end{thm}


\subsection{GSURE for $\lun$-Analysis Regularization} 
\label{sub:contrib-risk}

We now specialize the previous results to the case where the estimator $\solTY$ is a solution of \lasso; i.e. $\solTY=\solS$ and $\muTY=\mus$. For notational clarity and to highlight the dependency of $\dim(\GJ)$ on $y$, for $y \not\in \Hh_{\cdot,\lambda}$, we write $d(y) = \dim(\GJ)$ where $J$ is the \dcosup of any solution $\solS$ such that \eqref{eq:hj} holds. We then obtain the following corollary as a consequence of Theorems~\ref{thm:dof} and \ref{thm:gsure}.

\begin{cor}\label{cor:dof}
  Let $y = \Phi x_0 + w$ with $w \sim \Nn(0,\sigma^2\Id_Q)$. Then $\mus$ is weakly differentiable and 
  \begin{align*}
    \GSURE_{\Phi}(\solS) = &
    \normd{y - \mus}^2
    - Q \si^2
    + 2 \si^2 d(y) ,\\
    \GSURE_\Pi(\solS) = & \normd{x_{\ML}(y) - \Pi \solS}^2
    - \si^2 \tr( (\Phi \Phi^*)^{+} )
    + 2 \si^2 \tr( \Pi \AJ ) ,\\
    \GSURE_\Id(\solS) = & \normd{x_{\ML}(y) - \solS}^2
    - \si^2 \tr( (\Phi^* \Phi)^{-1} )
    + 2 \si^2 \tr( \AJ ) ~ .
  \end{align*}
  Moreover,
  $d(y)$ is an unbiased estimator of the DOF of \lasso, i.e.
  \begin{equation*}
    \DOF_{\lambda} = \DOF_{\lambda}^{\Id} = \EE_w d(y) .
  \end{equation*}
\end{cor}
In particular, this result states that $\dim(\GJ)$ is an unbiased estimator of the DOF of \lasso response without requiring any assumption to ensure uniqueness of $\solS$. This DOF estimator formula is valid everywhere except on a set of (Lebesgue) measure zero.

\medskip

Building upon Theorems~\ref{thm:dof} and \ref{thm:reliability}, we derive the relative reliability of the $\GSURE$ for \lasso, and show that it decays with the number of measurements at the rate $O(1/Q)$.
\begin{cor}\label{cor:reliability}
  Let $A \in \RR^{M \times Q}$ and $y = \Phi x_0 + w$ with $w \sim \Nn(0,\sigma^2\Id_Q)$.
  Then
  \begin{align*}
    \EE_w\left[ \left(\frac{\GSURE^A(\solS) - \SE^A(\solS)}{Q \sigma^2}\right)^2\right]
    = O\left(\frac{\norm{A}^4}{Q}\right)~.
  \end{align*}
  where $\norm{A}$ is the spectral norm of $A$.
  In particular, if $\norm{A}$ is independent of $Q$, the decay rate of the relative reliability is $O(1/Q)$.
\end{cor}


\subsection{Numerical Considerations} 
\label{sub:contrib-numeric}

The remaining obstacle faced when implementing the GSURE formulae of Corollary~\ref{cor:dof} is to compute the divergence term, i.e. the last trace term as given by $\dev{\lambda}{A}{y}=\tr\pa{ A \Phi \AJ \Phi^* A^* }$ (see Definition~\ref{def:gsureA}). However, for large scale-data as in image and signal processing, the computational storage required for the matrix in the argument of the trace would be prohibitive. Additionally, computing $\AJ$ can only be reasonably afforded for small data size. Fortunately, the structure of $\dev{\lambda}{A}{y}$ and the definition of $\AJ$ allows to derive an efficient and principled way to compute the trace term. This is formalized in the next result. 

\begin{prop}\label{prop:gsure_computation}
  One has
  \begin{equation}\label{eq-prop-calcul-1}
    \dev{\lambda}{A}{y} =
    \EE_Z( \dotp{\nu(Z)}{\Phi^* A^* A Z} )
  \end{equation}
  where $Z \sim \Nn(0,\Id_P)$,
  and where for any $z \in \RR^P$, $\nu = \nu(z)$ solves the following linear system
  \begin{equation}\label{eq-prop-calcul-2}
    \begin{pmatrix}
      \Phi^* \Phi & D_J \\
      D_J^* & 0
    \end{pmatrix}
    \begin{pmatrix}
      \nu \\ \tilde \nu
    \end{pmatrix}
    =
    \begin{pmatrix}
      \Phi^* z \\ 0
    \end{pmatrix}~.
  \end{equation}
\end{prop}

In practice, the empirical mean estimator is replaced for the expectation in \eqref{eq-prop-calcul-1}, hence giving
\begin{equation}
\label{eq:dofmean}
  \frac{1}{k}
  \sum_{i=1}^k \dotp{ \nu(z_i)}{\Phi^* A^* A z_i} \overset{\mathrm{WLLN}}{\longrightarrow} \dev{\lambda}{A}{y} ~,
\end{equation}
for $k$ realizations $z_i$ of $Z$, where WLNN stands for the Weak Law of Large Numbers. Consequently, the computational bulk of computing an estimate of $\dev{\lambda}{A}{y}$ is invested in solving for each $\nu(z_i)$ the symmetric linear system \eqref{eq-prop-calcul-2} using e.g. a conjugate gradient solver.



%% file: sections/pw.tex
\section{Relation to Other Works} 
\label{sec:related}

\subsection{Local variations} 
\label{sub:local_variations}

The local behavior of $\solS$ as a function of $\la$ is already known in the $\lun$-synthesis case, both for the case where $\Phi$ is full rank \cite{osborne2000lasso,osborne2000new}, and $Q < N$ \cite{donoho2008fast}. Our local affine parameterization in Theorem~\ref{thm:local} generalizes these results to the analysis case regardless of the number of measurements. Our result also goes beyond the work of~\cite{tibshirani2011solution} which investigates the overdetermined case with an $\lun$-analysis regularization and develops a homotopy algorithm.

\subsection{Degrees of freedom}
\label{sub:pwdof}

In the synthesis overdetermined case with full rank $\Phi$, \cite{zou2007degrees} showed that the number of nonzero coefficients is an unbiased estimate for the degrees of freedom of \lasso. This was generalized to an arbitrary $\Phi$ in \cite{kachour-dof-preprint}. Corollary~\ref{cor:dof} encompasses these results as special cases by taking $D=\Id$.

For the $\lun$-analysis regularization with full rank $\Phi$, Tibshirani and Taylor \cite{tibshirani2011solution} showed that $\DOF_{\lambda}=\EE_w \dim(\GJ)$, where $J$ is the \dcosup of the unique solution to \lasso. This is exactly the assertion of Corollary~\ref{cor:dof}, since \eqref{eq:hj} is in force when $\rank(\Phi)=N$.

While a first version of this paper was submitted, it came to our attention that Tibshirani and Taylor~\cite[Theorem 3]{tibshirani2012dof}~recently and independently developed an unbiased estimator of the DOF for \lasso that covers the case where $Q < N$. More precisely, they showed that $\dim(\Phi(\GJ))$ is an unbiased estimator of $\DOF(\lambda)$, where $J$ is the \dcosup of {\textit{any}} solution to \lasso. This coincides with Corollary~\ref{cor:dof} when $J$ satisfies \eqref{eq:hj}. Their proof however differs from ours, and in particular, its does not study directly the local behavior of $\solS$ as a function of $y$ or $\lambda$ (Theorem \ref{thm:local}).

\subsection{Generalized Stein Unbiased Risk Estimator}
\label{sub:gsure}

In \cite{eldar-gsure}, the author derived expressions equivalent to $\GSURE_\Pi$ and $\GSURE_\Id$ up to a constant which does not depend on the estimator. However, her expressions were developed separately, whereas we have shown that these GSURE estimates originate from a general result stated in Theorem~\ref{thm:gsure}. Another distinction between our work and \cite{eldar-gsure} lies in the assumptions imposed. 
The author \cite{eldar-gsure} supposes $\solTY$ to be a weakly differentiable function of $\Phi^* y / \sigma^2$.
In contrast, we just require that the prediction $y \mapsto \muTY$ (a single-valued map) is weakly differentiable, as classically assumed in the SURE theory.

\medskip

Indeed, let $u = \Phi^* y / \sigma^2$, and define $\solTY = z_\theta^\star(u)$. Assume that $u \mapsto z_\theta^\star(u)$ is weakly differentiable (and a fortiori a single-valued mapping).

When $\Phi$ is rank deficient, \cite{eldar-gsure} proves unbiasedness of the following estimator of the projection risk
\begin{align*}
  \GSURE_\Pi^{\text{(Eldar)}}(z_\theta^\star(u)) =&
  \normd{\Pi x_0 }^2 + \normd{\Pi z_\theta^\star(u) }^2
  - 2 \dotp{z_\theta^\star(u)}{x_{\ML}(y)} \\
  &+ 2 \tr \pa{\Pi \frac{\partial z_\theta^\star(u)}{\partial u} }.
\end{align*}
Since by assumption $\frac{\partial \Phi z_\theta^\star(u)}{\partial u}=\Phi \frac{\partial z_\theta^\star(u)}{\partial u}$, and using the chain rule, the following holds
\begin{align*}
  \sigma^2 \tr \pa{(\Phi\Phi^*)^+ \frac{\partial \muTY}{\partial y} }
  =
  \sigma^2 \tr \pa{(\Phi\Phi^*)^+ \frac{\partial \Phi z_\theta^\star(u)}{\partial u} \frac{\partial u}{\partial y} }
  =
  \tr \pa{\Pi \frac{\partial z_\theta^\star(u)}{\partial u} }
\end{align*}
whence it follows that
\begin{align*}
  \GSURE_\Pi(\solTY) - \GSURE_\Pi^{\text{(Eldar)}}(\solTY)
  =&
  \normd{x_{\ML}(y)}^2 - \normd{\Pi x_0 }^2 \\
  &- \sigma^2 \tr \!\pa{ (\Phi \Phi^*)^{+} } ~.
\end{align*}

A similar reasoning when $\Phi$ has full rank leads to
\begin{align*}
  \GSURE_\Id(\solTY) - \GSURE_\Id^{\text{(Eldar)}}(\solTY)
  =&
  \normd{x_{\ML}(y)}^2 - \normd{x_0 }^2 \\
  &- \sigma^2 \tr \!\pa{ (\Phi^* \Phi)^{-1} } ~.
\end{align*}

Both our estimators and those of \cite{eldar-gsure} are unbiased, but they do not have necessarily the same variance. Given that they only differ by terms that do not depend on $\solTY$, and in particular on the parameter (here $\theta$), selecting the latter by minimizing our GSURE expressions or those of \cite{eldar-gsure} is expected to lead to the same results. 

Let us finally mention that in the context of deconvolution, $\GSURE_\Pi$ boils down to the unbiased estimator of the projection risk obtained in~\cite{pesquet-deconv}.

\subsection{Numerical computation of the GSURE}
\label{sub:numgsure}

In least-squares regression regularized by a sufficiently smooth penalty term, the DOF can be estimated in closed-form~\cite{solo1996sure}. However even in such simple cases, the computational load and/or storage can be prohibitive for large-scale data.

To overcome the analytical difficult for general non-linear estimators, when no closed-form expression is available, first attempts developed bootstrap-based (asymptotically) unbiased estimators of the DOF~\cite{efron2004estimation}. Ye~\cite{ye1998measuring} and Shen and Ye~\cite{shen2002adaptive} proposed a data perturbation technique to approximate the DOF (and the SURE) when its closed-form expression is not available or numerically expensive to compute. For denoising, a similar Monte-Carlo approach has been used in~\cite{ramani2008montecarlosure} where it was applied to total-variation denoising, wavelet soft-thresholding, and Wiener filtering/smoothing splines.

Alternatively, an estimate can be obtained by recursively differentiating the sequence of iterates that converges to a solution of the original minimization problem. Initially, it has been proposed by~\cite{vonesch2008sure}, and then refined in~\cite{giryes-proj-gsure}, to compute the GSURE of sparse synthesis regularization by differentiating the sequence of iterates of the forward-backward splitting algorithm. 
We have recently proposed a generalization of this methodology to any proximal splitting algorithm, and exemplified it on $\lun$-analysis regularization
including the isotropic total-variation regularization, and $\lun-\ldeux$ synthesis regularization which promotes block sparsity  \cite{deledalle2012proximalncmip}.

In our case, we have shown that the computation of a good estimator of the DOF, and therefore of $\GSURE^{A}$ for various risks, boils down to solving linear systems. This is much more efficient than the previous general-purpose iterative methods that are computationally expensive.


%% file: sections/examples.tex

\begin{figure}[t]
  \centering
  \subfigure[$y$]{\includegraphics[width=0.325\linewidth,viewport=128 64 384 197,clip]{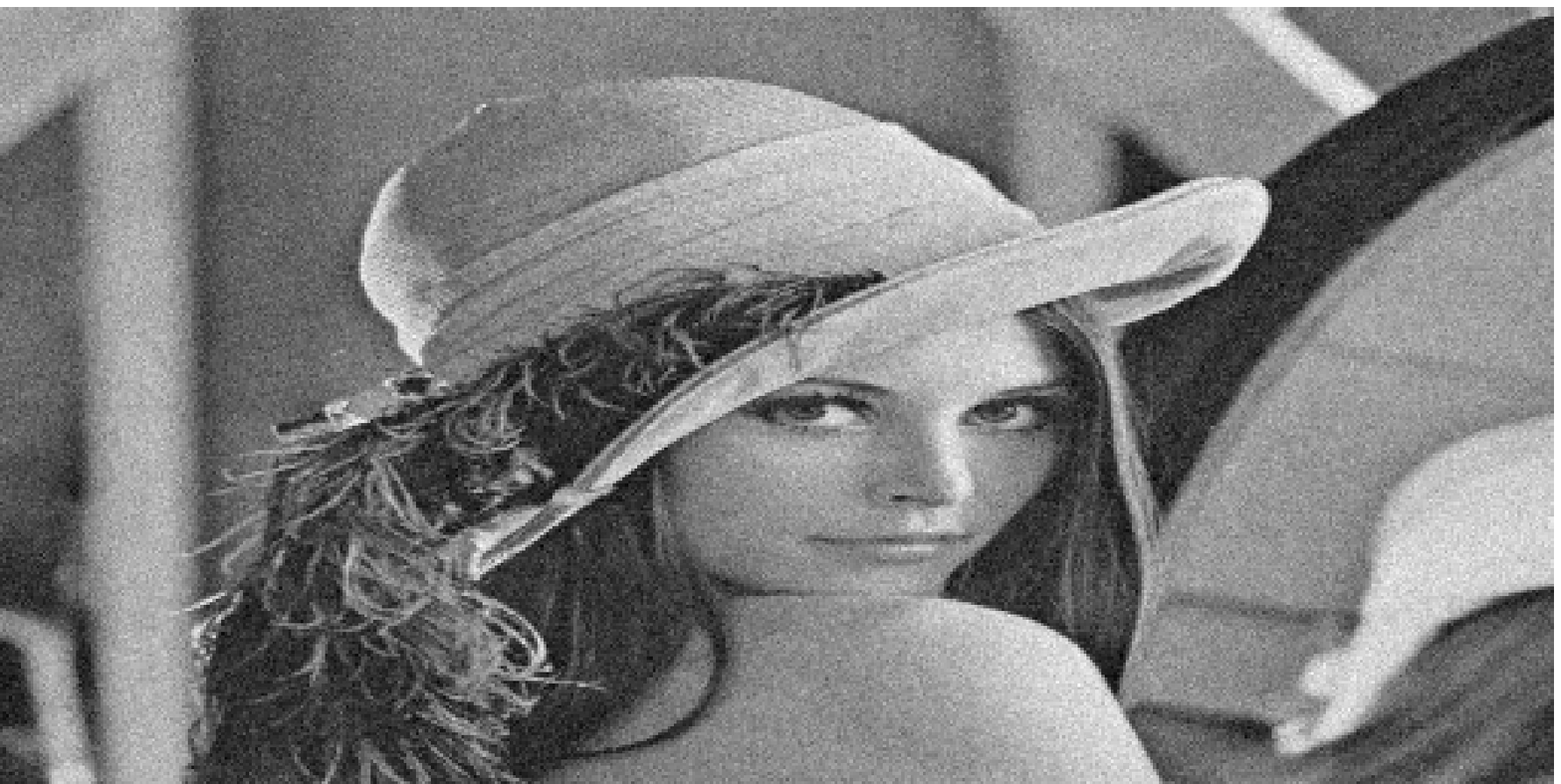}}
  \subfigure[$\solS$ at the optimal $\lambda$]{\includegraphics[width=0.325\linewidth,viewport=128 128 384 384,clip]{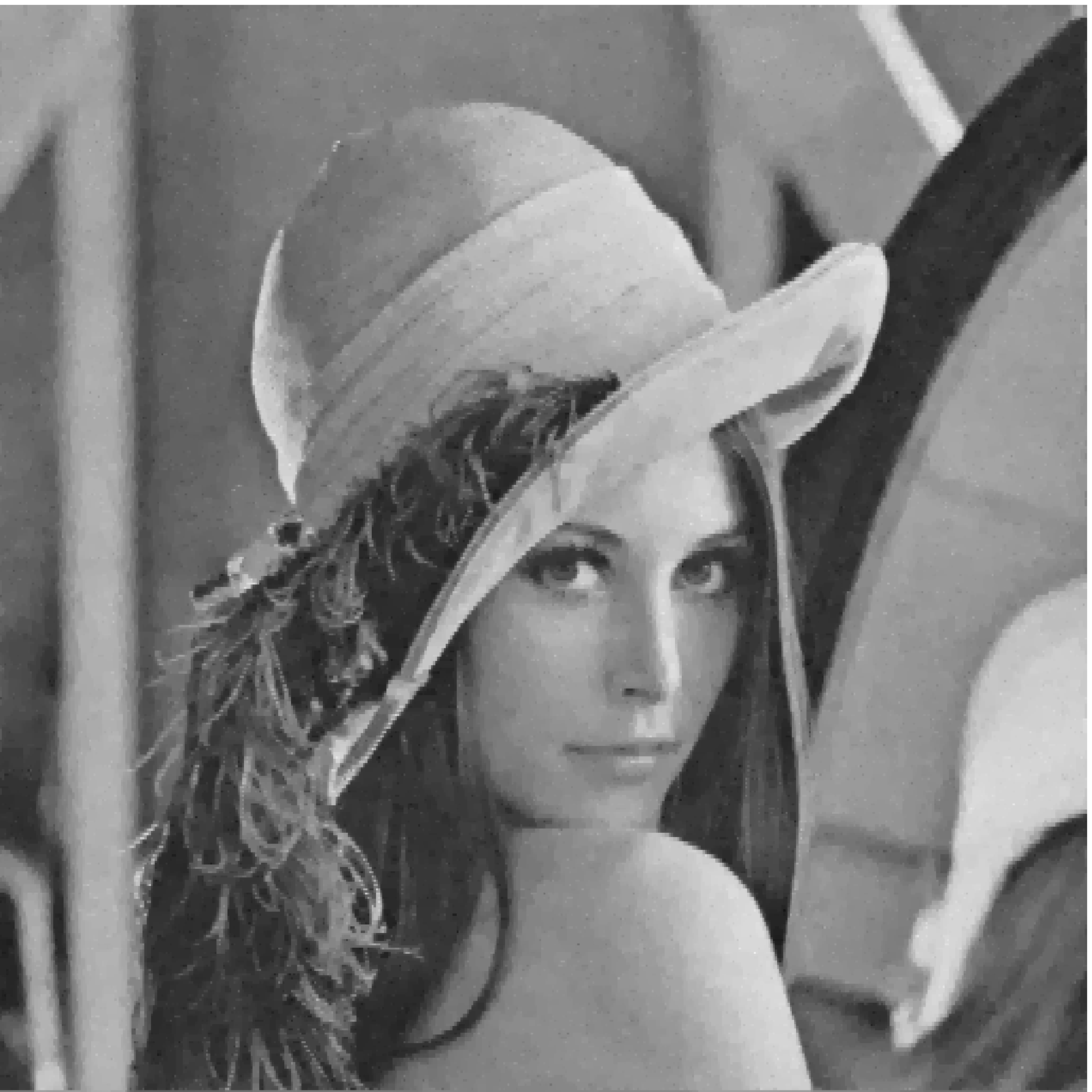}}
  \subfigure[$x_0$]{\includegraphics[width=0.325\linewidth,viewport=128 128 384 384,clip]{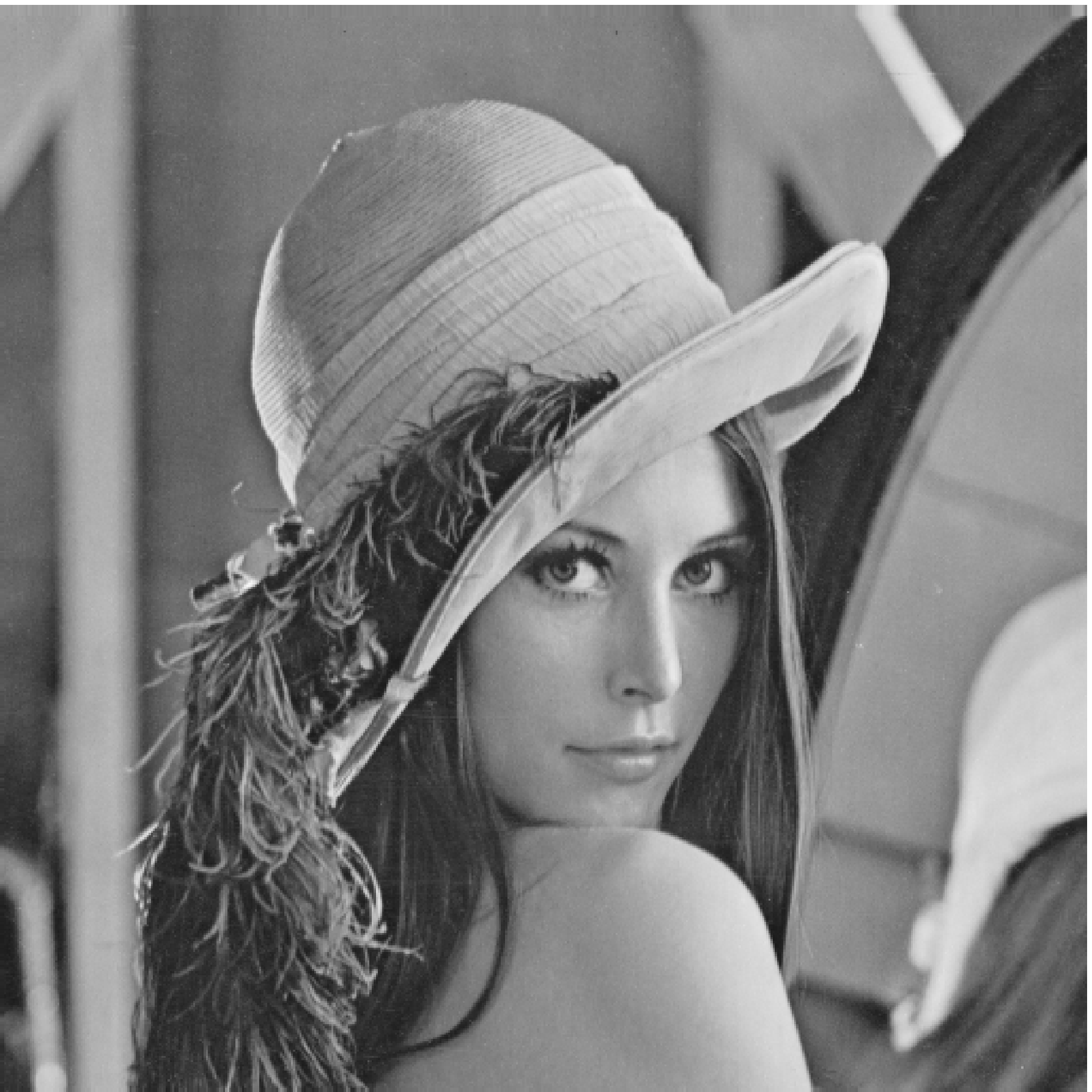}}\\
  \subfigure[]{\includegraphics[width=0.6\linewidth]{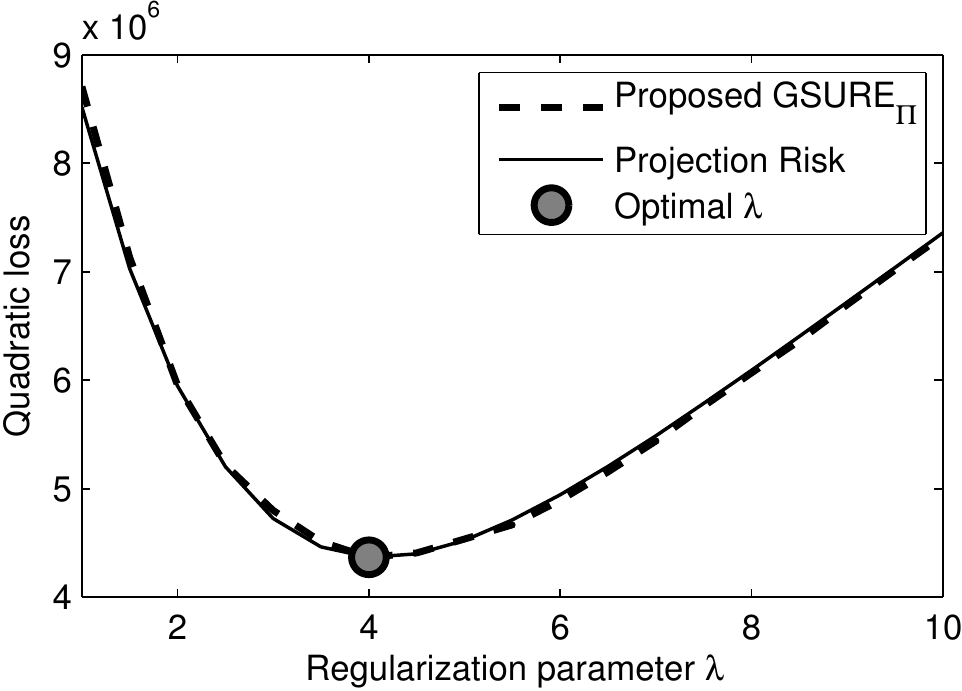}}
  \caption{
    Illustration of the selection of $\lambda$ by minimizing $\GSURE_{\Pi}$ in a
    super-resolution problem ($Q/N = 0.5$) with anisotropic total variation
    regularization.
    (a) The observed image $y$.
    (b) A solution $\solS$ of \lasso at the optimal $\lambda$ (the one minimizing $\GSURE_{\Pi}$).
    (c) The underlying true image $x_0$.
    (d) Projection risk $\risk_{\Pi}$ and its $\GSURE_{\Pi}$ estimate
    obtained from \eqref{eq:dofmean} using $k = 1$ random realization.
  }
  \label{fig:se_tv}
  \vspace{3pt}
\end{figure}

\begin{figure}[t]
  \centering
  \subfigure[$x_{\ML}$]{\includegraphics[width=0.325\linewidth]{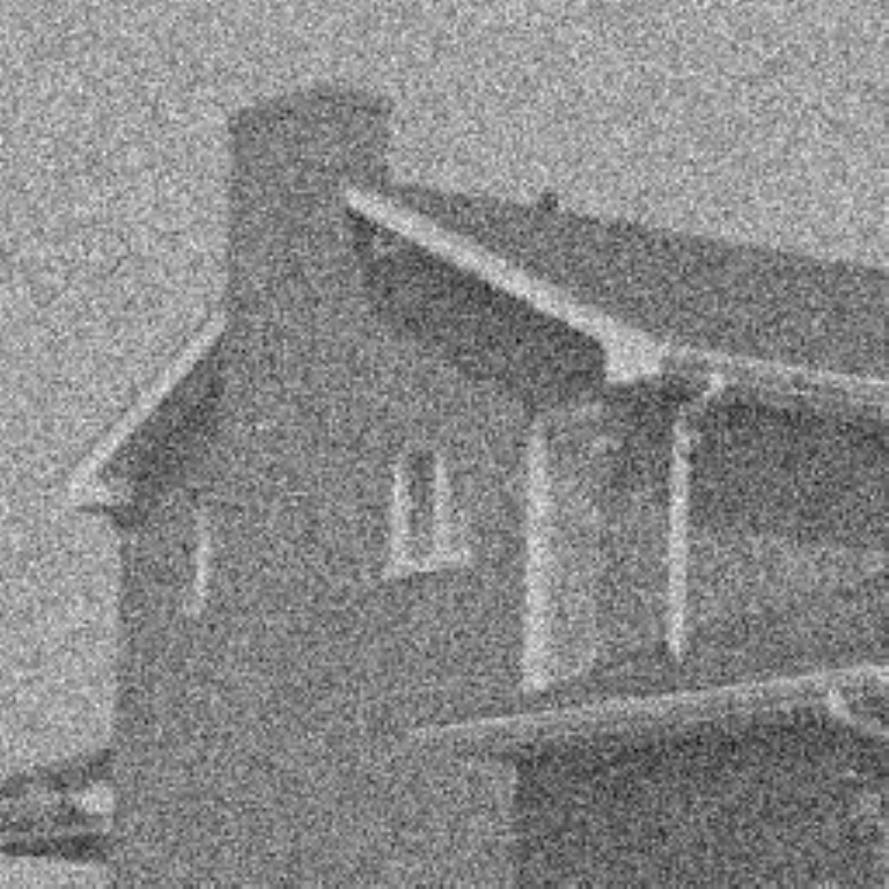}}
  \subfigure[$\solS$ at the optimal $\lambda$]{\includegraphics[width=0.325\linewidth]{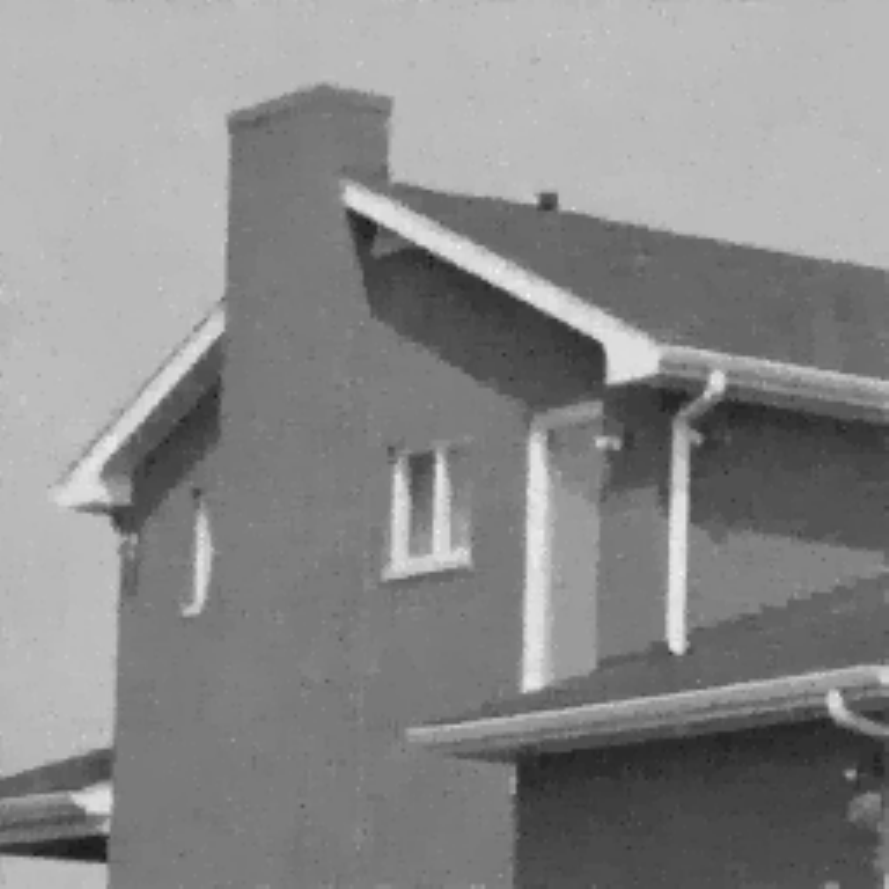}}
  \subfigure[$x_0$]{\includegraphics[width=0.325\linewidth]{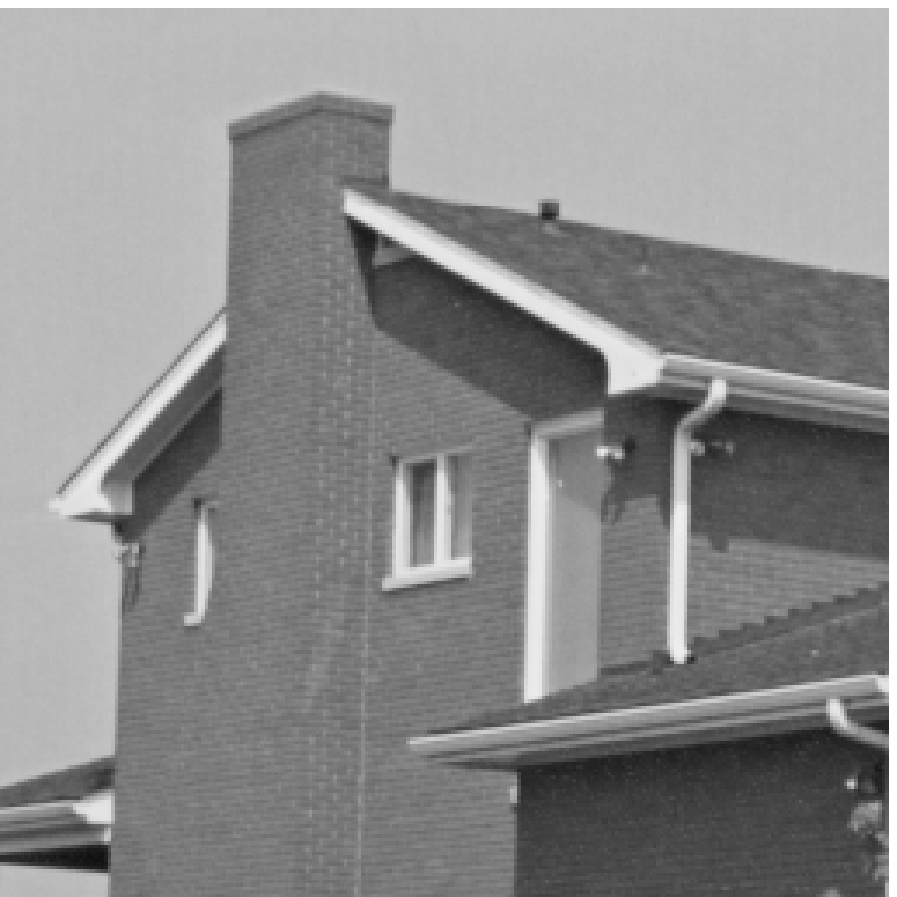}}\\
  \subfigure[]{\includegraphics[width=0.6\linewidth]{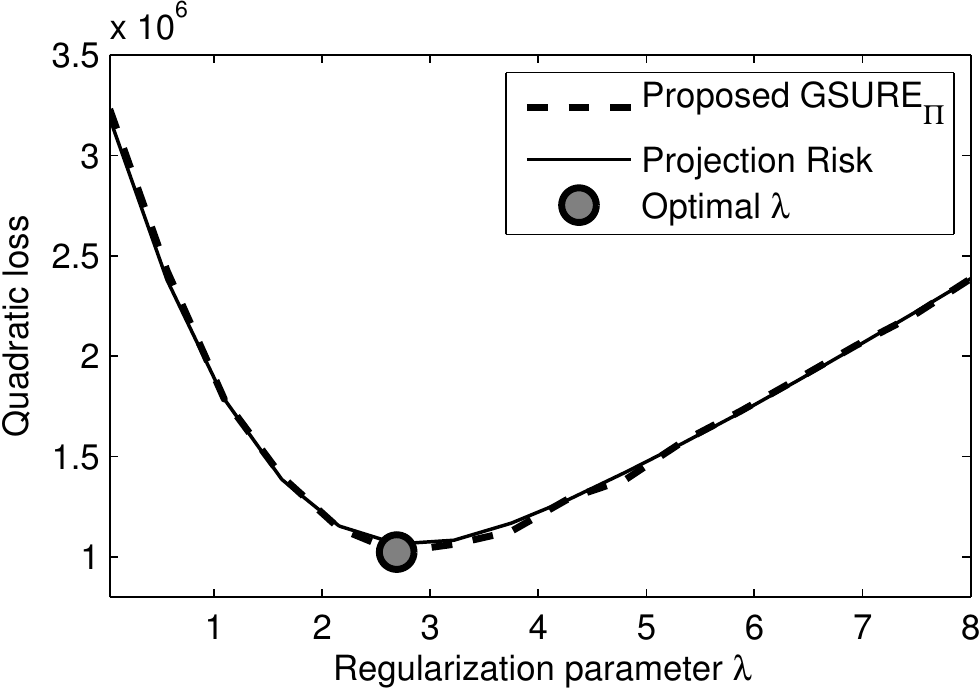}}
  \caption{
    Illustration of the selection of $\lambda$ by minimizing $\GSURE_{\Pi}$ in a
    compressed sensing problem ($Q/N = 0.5$)
    by an $\lun$-analysis regularization in 
    a shift-invariant Haar wavelet dictionary.
    (a) The MLE $x_{\ML}$.
    (b) A solution $\solS$ of \lasso at the optimal $\lambda$ (the one minimizing $\GSURE_{\Pi}$).
    (c) The underlying true image $x_0$.
    (d) Projection risk $\risk_{\Pi}$ and its $\GSURE_{\Pi}$ estimate
    obtained from \eqref{eq:dofmean} using $k = 1$ random realization.
  }
  \label{fig:analysis}
  \vspace{3pt}
\end{figure}

\section{Numerical Experiments} 
\label{sec:examples}

In this section, we exemplify the usefulness of our $\GSURE$ estimator which can serve as a basis for automatically tuning the value of $\la$. This is achieved by computing, from a single realization of the noise $w \sim \mathcal{N}(0,\si^2\Id)$, the parameter $\la$ that minimizes the value of $\GSURE$ when solving \lasso from $y=\Phi x_0 + w$ for various scenarios on $\Phi$ and $x_0$.

\subsection{Computing Minimizers} 
\label{sub:scheme}

\paragraph{Denoising}
Although it is convex, solving problem $\lasso$ is rather challenging given its non-smoothness.
In the case where $\Phi = \Id$, the objective functional of \lasso is strictly convex, and one can compute its unique solution $\solS$ by solving an equivalent equivalent Fenchel-Rockafellar dual problem \cite{chambolle2004algorithm}
\begin{equation*}
  \solS = y + D \alpha_{\lambda}^\star(y)
  \qwhereq 
  \alpha_{\lambda}^\star(y)  \in \uArgmin{\norm{\alpha}_\infty \leq \lambda} \norm{y + D\alpha}_2^2 .
\end{equation*}
This dual problem can be solved using e.g.~projected gradient descent or a multi-step accelerated version.

\paragraph{General Case}
The proximity operator of $x \mapsto \norm{D^* x}_1$ is not computable in closed-form for an arbitrary dictionary $D$.
This precludes the use of popular iterative soft-thresholding
(actually the forward-backward proximal splitting) without sub-iterating.
We therefore appeal to a more elaborate primal-dual splitting algorithm. 
We use in our numerical experiments the relaxed Arrow-Hurwicz algorithm as revitalized recently in \cite{chambolle2011first}.
This algorithm achieves full splitting where all operators are applied separately: the proximity operators of $g \mapsto \tfrac{1}{2}\normd{y - g}^2$ and $u \mapsto \la \norm{u}_1$ (which are known in closed-form), and the linear operators $\Phi$ and $D$ and their adjoints. To cast \lasso in the form required to apply this scheme, we can rewrite it as
\begin{equation*}
  \umin{x \in \RR^N} F(K(x))
  \qwhereq
  \choice{
    F(g,u) = \frac{1}{2}\norm{y-g}_2^2 + \lambda \normu{u} \\
    K(x) = ( \Phi x, D^* x ).
  }
\end{equation*}
Note that other algorithms could be equally applied to solve \lasso, e.g. \cite{raguet2011gfb,pustelnik2011prox,CombettesPesquet11}.

\subsection{Parameter Selection using the GSURE} 

\paragraph*{Super-Resolution with Total Variation Regularization}
In this example, $\Phi$ is a vertical sub-sampling operator of factor two (hence $Q/N = 0.5$).
The noise level has been set such that the observed image $y$ has a peak signal-to-noise ratio (PSNR) of 27.78 dB.
We used an anisotropic total variation regularization; i.e.~the sum of the $\lun$-norms of the partial derivatives in the first and second direction (not to be confused with the isotropic total variation).
Fig.~\ref{fig:se_tv}.d depicts the projection risk and its $\GSURE_\Pi$ estimate obtained from \eqref{eq:dofmean} with $k = 1$ as a function of $\la$.
The curves appear unimodal and coincide even with $k=1$ and a single noise realization.
Consequently, $\GSURE_\Pi$ provides a high-quality selection of $\la$ minimizing the projection risk.
Close-up views of the central parts of the degraded, restored (using the optimal $\la$), and true
images are shown in Fig.~\ref{fig:se_tv}(a)-(c) for visual inspection of the restoration quality.

\paragraph*{Compressed Sensing with Wavelet Analysis Regularization}
We consider in this example a compressed sensing scenario where $\Phi$ is a random partial DCT measurement matrix with an under-sampling ratio $Q/N = 0.5$.
The noise is such that input image $y$ has a PSNR set to 27.50 dB.
We took $D$ as the shift-invariant Haar wavelet dictionary with 3 scales.
Again, we estimate $\GSURE_\Pi$ with $k=1$ in \eqref{eq:dofmean}. The results observed on the super-resolution example are confirmed in this compressed
sensing experiment both visually and qualitatively, see Fig.~\ref{fig:analysis}.\\


%% file: sections/conclusion.tex
\section{Conclusion} 
\label{sec:conclusion}

In this paper, we studied the local behavior of solutions to $\lun$-analysis regularized inverse problems of the form \lasso. We proved that any minimizer $\solS$ of \lasso is a piecewise affine function of the observations $y$ and the regularization parameter $\lambda$. This local affine parameterization is completely characterized by the \dsup $I$ of $\solS$, i.e.~the set of indices of atoms in $D$ with non-zero correlations with $\solS$. As a byproduct, for $y$ outside a set of zero Lebesgue measure, the first-order variations of $\Phi\solS$ with respect to $y$ are obtained in closed-form.

We capitalized on these results to derive a closed-form expression of an unbiased estimator of the degrees of freedom of \lasso, and to objectively and automatically choose the regularization parameter $\lambda$ when the noise contaminating the observations is additive-white Gaussian.
Toward this goal, a unified framework to unbiasedly estimate several risk measures is proposed through the GSURE methodology. This encompasses several special cases such as the prediction, the projection and the estimation risk. A computationally efficient algorithm is designed to compute the GSURE in the context of $\lun$-analysis reconstruction. Illustrations on different imaging inverse problems exemplify the potential applicability of our theoretical findings.


%% file: sections/proofs.tex
\section{Proofs} 
\label{sec:proofs}

\input{sections/proofs_intro.tex}
\input{sections/proofs_local.tex}
\input{sections/proofs_dof.tex}
\input{sections/proofs_gsure.tex}


%% file: sections/proofs_intro.tex

Throughout, we use the shorthand notation $\obj$ for the objective function in \lasso
\begin{equation*}
  \obj(x) = \minform .
\end{equation*}
We remind the reader that condition \eqref{eq:H0} is supposed to hold true in all our statements.

\subsection{Preparatory lemmata}
The following key lemma will be central in our proofs. It gives the first order necessary and sufficient optimality conditions for the analysis variational problem \lasso.
\begin{lem}\label{lem:first-order}
  A vector $\solS$ is a solution of \lasso if, and only if, there exists $\sigma \in \RR^{\abs{J}}$, where $J$ is the \dcosup of $\solS$, such that
  \begin{equation}\label{eq:first-order}
    \sigma \in \Sig(\solS)
  \end{equation}
  with
  \begin{multline}\label{eq:first-order-explicit}
    \Sig(\solS) = \Big\{ \sigma \in \RR^{\abs{J}} \setminus  \Phi^*(\Phi \solS - y) + \lambda D_I s_I + \lambda D_J \sigma = 0 \\
                   \qandq \normi{\sigma} \leq 1 \Big\} ,
  \end{multline}
  where $I = J^c$ is the \dsup of  $\solS$ and $s = \sign(D^* \solS)$.
\end{lem}
\begin{proof}
   The subdifferential of a real-valued proper convex function $F : \RR^N \to \RR \cup \{\infty\}$ is denoted $\partial F$. From standard convex analysis, we recall of $\partial F$ at a point $x$ in the domain of $F$ 
  \begin{equation*}
    \partial F(x) = \enscond{ g \in \RR^N \!}{\!\forall z \in \RR^N\!, F(z) \! \geq \! F(x) \!+\! \dotp{g}{z - x}} .
  \end{equation*}
  It is clear from this definition that $\solS$ is a (global) minimizer of $F$ if, and only if, $0 \in \subd F(x)$.
  By classical subdifferential calculus, the subdifferential of $\obj$ at $x$ is the non-empty convex compact set
  \begin{equation*}
    \subd \obj(x) = \enscond{\Phi^* (\Phi x - y) + \lambda D u}{u \in \RR^N: u_I \!=\! \sign(D^* x)_I \text{ and } \normi{u_J} \leq 1} .
  \end{equation*}
  Therefore $0 \in \subd \obj(\solS)$ is equivalent to the existence of $u \in \RR^N$ such that $u_I = \sign(D^* \solS)_I$ and $\normi{u_J} \leq 1$ satisfying
  \begin{equation*}
    \Phi^* (\Phi \solS - y) + \lambda D u = 0 .
  \end{equation*}
  Taking $\sigma=u_J$, this is equivalent to $\sigma \in \Sig(\solS)$.
\end{proof}

\medskip

The following lemma gives an implicit equation satisfied by any (non necessarily unique) minimizer $\solS$ of \lasso.
\begin{lem}\label{lem:sol}
 Let $\solS$ be a solution of $\lasso$.
 Let $I$ be the \dsup and $J$ the \dcosup of $\solS$ and $s = \sign(D^* \solS)$.
 We suppose that \eqref{eq:hj} holds.
 Then, $\solS$ satisfies
 \begin{equation}\label{eq:local-expression}
   \solS = \AJ \Phi^* y - \lambda \AJ D_I s_I .
 \end{equation}
\end{lem}
\begin{proof}
  Owing to the first order necessary and sufficient optimality condition (Lemma \ref{lem:first-order}), there exists $\sigma \in \Sig(\solS)$ satisfying
  \begin{equation}\label{eq:first-order-x}
    \Phi^* (\Phi \solS - y) + \lambda D_I s_I + \lambda D_J \sigma = 0 .
  \end{equation}
  By definition, $\solS \in \GJ=(\Im D_J)^\bot$.
  We can then write $\solS = \UJ \alpha$ for some $\alpha \in \RR^{\dim(\GJ)}$.
  Since $\UJ^* D_J = 0$, multiplying both sides of \eqref{eq:first-order-x} on the left by $\UJ^*$, we get
  \begin{equation*}
    \UJ^* \Phi^* (\Phi \UJ \alpha - y) + \lambda \UJ^* D_I s_I = 0 .
  \end{equation*}
  Since $\UJ^* \Phi^* \Phi \UJ$ is invertible, the implicit equation of $\solS$ follows immediately.
\end{proof}

\medskip

Suppose now that a vector satisfies the above implicit equation. The next lemma derives two equivalent necessary and sufficient conditions to guarantee that this vector is actually a solution to \lasso.
\begin{lem}\label{lem:soldeux}
  Let $y \in \RR^Q$, let $J$ a \dcosup such that \eqref{eq:hj} holds and let $I = J^c$.
  Suppose that $\solSp$ satisfies
  \begin{equation*}
    \solSp = \AJ \Phi^* y  - \lambda \AJ D_I s_I .
  \end{equation*}
  where $s = \sign(D^* \solSp)$.
  Then, $\solSp$ is a solution of \lasso if, and only if, there exists $\sigma \in \RR^{|J|}$ satisfying one of the following equivalent conditions
  \begin{equation}\label{eq:cond-imp}
    \sigma - \Omega^{[J]} s_I + \frac{1}{\lambda} \BJ^{[J]} y \in \Ker D_J \qandq \normi{\sigma} \leq 1 ,
  \end{equation}
  or
  \begin{equation}\label{eq:forme-soldeux}
    \tilde \BJ^{[J]} y - \lambda \tilde \CJ^{[J]} s_I + \lambda D_J \sigma = 0 \qandq \normi{\sigma} \leq 1 ,
  \end{equation}  
  where $\tilde \CJ^{[J]} = (\Phi^* \Phi \AJ - \Id)D_I$, $\tilde \BJ^{[J]} = \Phi^* (\Phi \AJ \Phi^* - \Id)$, $\Omega^{[J]} = D_J^+ \tilde \Omega^{[J]}$ and $\BJ^{[J]} = D_J^+ \tilde \BJ^{[J]}$.
\end{lem}
\begin{proof}
  First, we observe that $\solSp \in \GJ$.
  According to Lemma \ref{lem:first-order}, $\solSp$ is a solution of \lasso if, and only if, there exists $\sigma \in \Sig(\solSp)$.
  Since \eqref{eq:hj} holds, $\AJ$ is properly defined.
  We can then plug the assumed implicit equation in \eqref{eq:first-order-explicit} to get
  \begin{equation*}
    \Phi^* (\Phi \AJ \Phi^* y  - \lambda \Phi \AJ D_I s_I - y) + \lambda D_I s_I + \lambda D_J \sigma = 0 .
  \end{equation*}
  Rearranging the terms multiplying $y$ and $s_I$, we arrive at
  \begin{equation*}
    \Phi^*(\Phi \AJ \Phi^* - \Id) y - \lambda (\Phi^* \Phi \AJ - \Id)D_I s_I + \lambda D_J \sigma = 0 .
  \end{equation*}
  This shows that $\solS$ is a minimizer of \lasso if, and only if
  \begin{equation*}
    \tilde \BJ^{[J]} y - \lambda \tilde \CJ^{[J]} s_I + \lambda D_J \sigma = 0 \qandq \normi{\sigma} \leq 1 .
  \end{equation*}

  To prove the equivalence with \eqref{eq:forme-soldeux}, we first note that $\UJ^* \tilde \CJ^{[J]}=0$ implying that $\Im(\tilde \CJ^{[J]}) \subseteq \Im(D_J)$. Since $D_JD_J^+$ is the orthogonal projector on $\Im(D_J)$, we get $\tilde \CJ^{[J]}=D_JD_J^+\tilde \CJ^{[J]}=D_J\CJ^{[J]}$.
  With a similar argument, we get $\tilde \BJ^{[J]} = D_J \BJ^{[J]}$.
  Hence, the existence of $\sigma \in \Sig(\solSp)$ such that $\normi{\sigma} \leq 1$ is equivalent to
  \begin{equation*}
    D_J \sigma = D_J \CJ^{[J]} s_I - \frac{1}{\lambda} D_J \BJ^{[J]} y \qwhereq \normi{\sigma} \leq 1 ,
  \end{equation*}
  which in turn is equivalent to
  \begin{equation*}
    \sigma - \CJ^{[J]} s_I + \frac{1}{\lambda} \BJ^{[J]} y \in \Ker D_J \qwhereq \normi{\sigma} \leq 1 .
  \end{equation*}
\end{proof}

We now show that even if \lasso admits several solutions $\solS$, all of them share the same image under $\Phi$, which in turn implies that $y\mapsto\mus$ is a single-valued mapping.
\begin{lem}\label{lem:unique-image}
  If $x_1$ and $x_2$ are two minimizers of \lasso, then $\Phi x_1 = \Phi x_2$.
\end{lem}
\begin{proof}
  Let $x_1, x_2$ be two minimizers of \ref{eq:lasso-a}. Suppose that $\Phi x_1 \neq \Phi x_2$.
  Take $x_3 = \rho x_1 + (1-\rho)x_2$, $\rho \in (0,1)$.
  Strict convexity of $u \mapsto \normd{y - u}^2$ implies that
  \begin{equation*}
    \frac{1}{2} \normd{y - \Phi x_3}^2 < \frac{\rho}{2} \normd{y - \Phi x_1}^2 +  \frac{1-\rho}{2} \normd{y - \Phi x_2}^2  .
  \end{equation*}
  Jensen's inequality again applied to the $\lun$ norm gives
  \begin{equation*}
    \normu{D^* x_3} \leq \rho \normu{D^* x_1} + (1-\rho)\normu{D^* x_2} .
  \end{equation*}
  Together, these two inequalities yield $\obj(x_3) < \obj(x_1)$, which contradicts our initial assumption that $x_1$ is a minimizer of \lasso.
\end{proof}

%% file: sections/proofs_local.tex
\subsection{Proof of Theorem \ref{thm:local}} 
\label{sub:local}

\begin{proof}
  Let $(y,\lambda) \not\in \Hh$.
  By construction, the vector $\solSpLY$ obeys $D_J^* \solSpLY=0$.
  Accordingly, for $(\bar y, \bar \lambda)$ sufficiently close to $(y,\lambda)$, one has
  \begin{equation*}
    \sign(D^* \solSpLY) = \sign(D^* \solS).
  \end{equation*}
  Since $\xsoly$ is a solution of \lasso, using Lemmas \ref{lem:sol} and \ref{lem:soldeux}, there exists $\sigma$ such that
  \begin{equation}\label{eq:implicit}
    \tilde \BJ^{[J]} y - \lambda \tilde \CJ^{[J]} s_I + \lambda D_J \sigma = 0 \qandq \normi{\sigma} \leq 1 .
  \end{equation}
  Let us split $J = K \cup L$, $K \cap L = \emptyset$ such that $\normi{\sigma_{K}} = 1$ and $\normi{\sigma_{L}} < 1$. Note that $\sigma_K \in \{-1,1\}^{\abs{K}}$. \\

  We first suppose that $\Im \tilde \BJ^{[J]} \subseteq \Im D_{L}$.
  To prove that $\solSpLY$ is solution to $(\lassoP{\bar y}{\bar \lambda})$, we show that there exists $\bar \sigma$ such that $\normi{\bar{\sigma}} \leq 1$ and
  \begin{equation*}
    \tilde \BJ^{[J]} \bar y - \bar \lambda \tilde \CJ^{[J]} s_I + \bar \lambda D_K \bar{\sigma}_K + \bar \lambda D_{L} \bar{\sigma}_{L} = 0.
  \end{equation*}
  We impose that $\bar{\sigma}_K = \sigma_K$ and take $\bar{\sigma}_{L}$ as
  \begin{equation*}
    \bar{\sigma}_{L} = \sigma_{L} - \frac{1}{\lambda} D_{L}^+ \tilde \Pi^{[J]} \pa{ \frac{\lambda - \bar \lambda}{\bar \lambda} y + \frac{\lambda}{\bar \lambda} (\bar y - y) } .
  \end{equation*}
  Hence,
  \begin{eqnarray*}
    & &   \tilde \BJ^{[J]} \bar y 
        - \bar \lambda \tilde \CJ^{[J]} s_I 
        + \bar \lambda D_J \bar \sigma \\
    & = & \tilde \BJ^{[J]} \bar y
        - \bar \lambda \tilde \CJ^{[J]} s_I 
        + \bar \lambda D_K \sigma_K
        + \bar \lambda D_{L} \sigma_{L} \\
    &   & - D_{L} D_{L}^+ \frac{\bar \lambda}{\lambda} \tilde \Pi^{[J]} \pa{ \frac{\lambda - \bar \lambda}{\bar \lambda} y + \frac{\lambda}{\bar \lambda} (\bar y - y) } \\
    & = & \underbrace{\tilde \BJ^{[J]} y
        - \lambda \tilde \CJ^{[J]} s_I 
        + \lambda D_K \sigma_K
        + \lambda D_{L} \sigma_{L}}_{=0} \\
    &  & - \tilde \Pi^{[J]} (y - \bar y) 
         + (\lambda - \bar \lambda) \tilde \CJ^{[J]} s_I 
         - (\lambda - \bar \lambda) D_K \sigma_K
         - (\lambda - \bar \lambda) D_{L} \sigma_{L} \\
    &  & - D_{L} D_{L}^+ \frac{\bar \lambda}{\lambda} \tilde \Pi^{[J]} \pa{ \frac{\lambda     - \bar \lambda}{\bar \lambda} y + \frac{\lambda}{\bar \lambda} (\bar y - y) } ~.
  \end{eqnarray*}
  Since $\Im \tilde \BJ^{[J]} \subseteq \Im D_{L}$ and $D_L D_L^+$ is the orthogonal projector on $\Im(D_L)$, we have $\tilde \BJ^{[J]}=D_L D_L^+\tilde \BJ^{[J]}$.
  It follows that,
  \begin{eqnarray*}
    & &   \tilde \BJ^{[J]} \bar y 
        - \bar \lambda \tilde \CJ^{[J]} s_I 
        + \bar \lambda D_J \bar \sigma \\
    & = & \frac{\bar \lambda - \lambda}{\lambda} \left[ \tilde \BJ^{[J]} y - \lambda \tilde \CJ^{[J]} s_I + \lambda D_K \sigma_K + \lambda D_{L} \sigma_{L} \right] \\
    & = & 0 .
  \end{eqnarray*}
  Now, for $(\bar y,\bar \lambda)$ close enough to $(y,\lambda)$, we have
  \begin{equation*}
    \normi{\bar \sigma_{L}} = \normi{\sigma_{L} + \frac{1}{\lambda} D_{L}^+ \tilde \Pi^{[J]} \pa{ \frac{\bar \lambda - \lambda}{\bar \lambda} y + \frac{\lambda}{\bar \lambda} (y - \bar y) } }\leq 1 ,
  \end{equation*}
  whence we deduce that $\solSpLY$ is a solution of $(\lassoP{\bar y}{\bar \lambda})$.

  In fact, for $(y,\lambda) \not\in \Hh$, we inevitably have $\Im \tilde \BJ^{[J]} \subseteq \Im D_{L}$.
  Indeed, projecting \eqref{eq:implicit} on $\Gg_{L}$ gives
  \begin{equation*}
    0 = \Proj_{\Gg_{L}} \left( \tilde \BJ^{[J]} y - \lambda \tilde \CJ^{[J]} s_I + \lambda D_J \sigma  \right) = \Proj_{\Gg_{L}} \left( \tilde \BJ^{[J]} y - \lambda \tilde \CJ^{[J]} s_I + \lambda D_K \sigma_K  \right) .
  \end{equation*}
  or equivalently
  \begin{equation*}
    \Proj_{\Gg_{L}}\tilde \BJ^{[J]} y = \Proj_{\Gg_{L}}\lambda(\tilde \CJ^{[J]} s_{I} - D_K \sigma_K) .
  \end{equation*}
  If $\Im \tilde \BJ^{[J]} \not\subseteq \Im D_{L}$, then $(y,\lambda) \in \Hh_{J,K,s_{I},\sigma_K}$, a contradiction. This concludes the proof.
\end{proof}

%% file: sections/proofs_dof.tex
\subsection{Proof of Theorem \ref{thm:dof}} 
\label{sub:dof}

\begin{proof}[Proof of (i)]
  First it is easy to see that $\Hh_{J,K,s_{J^c},\sigma_K}$ in Definition~\ref{defn:h} is a vector subspace of $\RR^Q \times \RR$.
  Moreover $\Hh_{J,K,s_{J^c},\sigma_K} \subseteq \ker \Proj_{\Gg_{J \setminus K}}B$, where $B=[\tilde \BJ^{[J]} ~~ -\tilde \CJ^{[J]}s_{J^c} + D_K \sigma_K]$.

  Now, fix $\lambda$. $\Hh_{\cdot,\lambda}$ is included in
  \begin{equation*}
    \tilde \Hh^\lambda = \bigcup_{\substack{J \subset \ens{1,\cdots,P} \\ \text{\eqref{eq:hj} holds} }}
          \bigcup_{\substack{K \subset J \\ \Im \PIJ \not\subset \Im D_{J \setminus K}}}
          \bigcup_{s_{J^c} \in \ens{-1,1}^{\abs{J^c}}}
          \bigcup_{\sigma_K \in \ens{-1,1}^{\abs{K}}}
          \tilde \Hh_{J,K,s_{J^c},\sigma_K},
  \end{equation*}
  where
  \begin{equation*}
    \tilde \Hh_{J,K,s_{J^c},\sigma_K}^\lambda = \enscond{y \in \RR^Q}{\Proj_{\Gg_{J \setminus K}}\PIJ y = \Proj_{\Gg_{J \setminus K}}\lambda(\CJJ s_{J^c} + D_K \sigma_K)} ,
  \end{equation*}
  Since $\Im \tilde \BJ^{[J]}  \not\subseteq \Im(D_{J \setminus K})$, $\Hh_{J,K,s_{J^c},\sigma_K}^\lambda$ is an affine subspace of $\RR^Q$ with $\dim(\tilde \Hh_{J,K,s_{J^c},\sigma_K}^\lambda) = \dim(\ker\Proj_{\Gg_{J \setminus K}}\PIJ) < Q$, where the inequality follows from the rank-nullity theorem and the fact that $\Gg_{J \setminus K}$ is a (nonempty) strict subspace of $\RR^Q$. Given that $\tilde \Hh^\lambda$ is a finite union of subspaces $\tilde \Hh_{J,K,s_{J^c},\sigma_K}^\lambda$ all strictly included in $\RR^Q$, $\tilde \Hh^\lambda$ has a Lebesgue measure zero and so does $\Hh_{\cdot,\lambda} \subseteq \tilde \Hh^\lambda$.
  
  Note that with a similar reasoning, one can show that $\Hh$ is also of zero Lebesgue measure using the fact that $B \not\subseteq \Im D_{J \setminus K}$ if $\tilde \BJ^{[J]} \not\subseteq \Im D_{J \setminus K}$ since  $\Im \tilde \BJ^{[J]} \subseteq \Im B$.
\end{proof}

\begin{proof}[Proof of (ii)]
  The proof of this statement is constructive.
  Denote by $\mathcal{M}_\lambda(y)$ the set of minimizers of \lasso. To lighten the notation, we drop the dependence on $y$ and $\lambda$ from $\solS \in \mathcal{M}_\lambda(y)$.

  \paragraph{\textbf{First step}} We prove the following statement
  \begin{align*}
    &\Big( \left( x^\star \in \mathcal{M}_\lambda(y) \wedge \neg (H_{\supp(D^* x^\star)^c}) \right)\\
    &\hspace{4em} \Longrightarrow
    \exists x_\lambda^{\star\star}(y) \in \mathcal{M}_\lambda(y) \,\wedge\, \supp(D^* x^{\star\star}) \subsetneq \supp(D^* x^\star) \Big),
  \end{align*}
  wher $\wedge$ and $\neg$ are respectively the logical conjunction and negation symbols. 
  In plain words, let $x^\star$ be a solution of \lasso.
  Suppose \eqref{eq:hj} does not hold where $J$ is the \dcosup of $x^\star$.
  We prove that there exists a solution $x_\lambda^{\star\star}(y)$ of \dsup strictly included in $I = J^c$.

  Since \eqref{eq:hj} does not hold, there exists $z \in \Ker \Phi$ with $z \neq 0$ and $D_J^* z = 0$.
  We define for every $t \in \RR$, the vector $v_t = x^\star + t z$.
  Denote $\Bb$ the subset of $\mathbb{R}$ defined by
  \begin{equation*}
    \Bb = \enscond{t \in \RR}{\sign(D^* v_t) = \sign(D^* x^\star) } ,
  \end{equation*}
  $\Bb$ is a non-empty convex set and $0\in\Bb$.
  Moreover for all
  $t\in\Bb$, $\partial \mathcal{L}_{y,\lambda}(v_t)=\partial
  \mathcal{L}_{y,\lambda}(x^\star)$. It then follows from Lemma~\ref{lem:first-order} that for all $t\in\Bb$,
  $v_t$ is a solution of \lasso.
  As a consequence, using Lemma~\ref{lem:unique-image}, we get
  \begin{equation*}
    \forall t \in \Bb, \quad \Phi v_t=\Phi x^\star \qandq \normu{D^* v_t}=\normu{D^* x^\star} .
  \end{equation*}
  Since
  $\underset{|t|\to\infty}{\lim}\normu{D^* v_t}=+\infty$, the set
  $\Bb$ is bounded.
  It is also an open set as a finite intersection of $P$ open sets corresponding to the solutions to $\sign((D^* x^\star)_i + t z_i) = \sign((D^* x^\star)_i))$.
  Hence, $\Bb$ is an open interval of $\RR$ which contains 0, i.e. there exist $t_1,t_0 \in \RR$ such that
  \begin{equation*}
    \Bb = ]t_1,t_0[ \qwhereq -\infty < t_1 < 0 \qandq 0 <t_0 < + \infty .
  \end{equation*}

  Since $t_0 \not\in \Bb$, the \dsup of $v_{t_0}$ is strictly included in $I$.
  Moreover by continuity,
  \begin{equation*}
    \Phi v_{t_0}=\Phi x^\star \qandq \normu{D^* v_{t_0}} = \normu{D^* x^\star} .
  \end{equation*}
  Hence, $v_{t_0}$ is a solution of \lasso of \dsup strictly included in $I$.

  \paragraph{\textbf{Second step}} We now prove our claim, i.e.
  \begin{equation*}
    \exists x^\star \in \mathcal{M}_\lambda(y) ~ \text{such that} ~ (H_{\supp(D^* x^\star)^c}) ~ \text{holds}.
  \end{equation*}
  Consider $(x_{(1)}^\star, \dots, x_{(P+1)}^\star) \in (\mathcal{M}_\lambda(y))^{P+1}$ such that for every $i \in \ens{1,\dots,P+1}$, the condition $(H_{J_i})$ does not hold for $J_i = \supp(D^* x_{(i)}^\star)^c$ and $J_1 \supsetneq J_2 \supsetneq \dots \supsetneq J_{P+1}$.
  Then, we have a strictly increasing sequence of $P+1$ subsets of $\ens{1,\dots,P}$ which is impossible.
  Hence, according to the first step of our proof, there exists $i \in \ens{1,\dots,P+1}$ such that $(H_{J_i})$ holds.
\end{proof}

\begin{proof}[Proof of (iii)]
  By virtue of statement (ii), there exists a solution $\solS$ of \lasso such that \eqref{eq:hj} holds.
  Let us consider this solution.
  Using Theorem~\ref{thm:local} for $\bar y$ close enough to $y$, we have
  \begin{equation*}
    \Phi \solSpY = \Phi \AJ \Phi^* \bar y - \lambda \Phi \AJ D_I s_I .
  \end{equation*}
  where $J$ is the \dcosup of $\solS$.
  Since $I$ (hence $J$) and $s_I$ are locally constant under the assumptions of the theorem, so is the vector $\lambda \Phi \AJ D_I s_I$, it follows that $\mu^\star_{\lambda}(\bar y) = \Phi \solSpY$ can be written as 
  \begin{equation*}
    \mu^\star_{\lambda}(\bar y) = \mus + \Phi \AJ \Phi^*(\bar y - y) ~,
  \end{equation*}
  whence we deduce
  \begin{equation*}
    \frac{\partial \mus}{\partial y} = \Phi \AJ \Phi^* .
  \end{equation*}
  Moreover, owing to statement (i), this expression is valid on $\RR^Q \setminus \Hh_{\cdot,\lambda}$, a set of full Lebesgue measure.
\end{proof}


%% file: sections/proofs_gsure.tex

\subsection{Proof of Theorem \ref{thm:gsure}}

We First recall Stein's lemma whose proof can be found in \cite{stein1981estimation}.

\begin{lem}[Stein's lemma] \label{lem:stein}
Let $y = \Phi x_0 + w$ with $w \sim \Nn(0,\sigma^2\Id_Q)$.
Assume that $g : y \mapsto g(y)$ is weakly differentiable (and a fortiori a single-valued mapping),
then
\begin{align*}
  \EE_w \dotp{w}{g(y)}
  =
  \si^2 \EE_w \tr\bra{\frac{\partial g(y)}{\partial y}} ~.
\end{align*}

\end{lem}

Let us now turn to the proof of Theorem \ref{thm:gsure}.
\begin{proof}
  Since $y \mapsto \muTY=\Phi \solTY$ is weakly differentiable, so is $A^* A \muTY$ and
  we have
  \begin{align*}
    \frac{\partial A^* A \muTY}{\partial y}
    =
    A^* A \frac{\partial \muTY}{\partial y} ~.
  \end{align*}
  Then,
  using Lemma~\ref{lem:stein}, we get
  \begin{align*}
    \EE_w \dotp{w}{A^* A \muTY}
    =
    \sigma^2
    \EE_w \tr \pa{A^* A \frac{\partial \muTY}{\partial y}}
    =
    \sigma^2 \EE_w \dev{\theta}{A}{y} ~.
  \end{align*}
  Using the decomposition $A y = A \Phi x_0 + A w$, we obtain
  \begin{align*}
    \EE_w\normd{A y - A \muTY}^2
    &=
    \EE_w\normd{A \Phi x_0 + A w}^2
    - 2 \EE_w \dotp{A \Phi x_0 + A w}{A \muTY}\\
    & \hspace{1cm}
    + \EE_w\normd{A \muTY}^2\\
    &=
    \EE_w\normd{A \Phi x_0}^2 + \si^2 \tr( A^* A )
    - 2 \EE_w \dotp{ A \Phi x_0}{A \muTY}\\
    &
    \hspace{1cm}
    - 2 \EE_w \dotp{w}{A^* A \muTY} 
    + \EE_w \normd{A \muTY}^2\\
    &=
  \EE_w\normd{ A \Phi x_0 - A \muTY}^2\\
  & \hspace{1cm}
  + \si^2 \tr( A^* A )
  - 2 \si^2 \EE_w \dev{\theta}{A}{y} ~.
\end{align*}
Moreover, $\sum_i \cov_w((A y)_i, (A \muTY)_i)\!=\!\EE_w\dotp{A w}{A \muTY}$,
which shows that $\dev{\theta}{A}{y}$ is indeed an unbiased estimator of $\DOF_\theta^A$.
\end{proof}

\subsection{Proof of Theorem \ref{thm:reliability}} 
\label{sub:reliability_dof}

\begin{proof}
Denote by $R^A$ the reliability of the GSURE for the estimator $\solTY$, i.e.
\begin{align*}
  R^A = \EE_w \left( \GSURE^A(\solTY) - \SE^A(\solTY) \right)^2.
\end{align*}
Let $Q^A(\solTY)$ be the quantity defined as
\begin{align*}
  Q^A(\solTY) = \normd{A \mu_0}^2 + \normd{A \muTY}^2 - 2 \dotp{A y}{A \muTY} + 2 \sigma^2 \dev{\theta}{A}{y}.
\end{align*}
We have
$\GSURE^A(\solTY) - Q^A(\solTY)
= \normd{A y}^2 - \EE_w \normd{A y}^2$, where
\begin{align*}
  \EE_w \normd{A y}^2 &= \normd{A \mu_0}^2 + \sigma^2 \tr(A^* A),\\
  \qandq
  \VV_w \normd{A y}^2 &=
  2 \sigma^4 \left(
  \tr[(A^* A)^2] +
  2 \frac{\normd{A^* A \mu_0}^2}{\sigma^2}
  \right).
\end{align*}
It results that
$\EE_w \left(\GSURE^A(\solTY) - Q^A(\solTY)\right) = 0$,
and hence
\begin{align*}
  \EE_w\left(Q^A(\solTY)\right) =
  \EE_w\left(\GSURE^A(\solTY)\right) =
  \EE_w\left(\SE^A\right).
\end{align*}
Remark that
$
  Q^A(\solTY) - \SE^A(\solTY) =
  2 \left(
  \sigma^2 \dev{\theta}{A}{y} - \dotp{A w}{A \muTY}
  \right).
$
We can now rewrite the reliability in the following form
\begin{align*}
  R^A
  =& \EE_w\left( \GSURE^A(\solTY) - Q^A(\solTY) + Q^A(\solTY) - \SE^A(\solTY) \right)^2\\
  =&
  \VV_w \normd{A y}^2
  \!+\! \EE_w \left(Q^A(\solTY) \!-\! \SE^A(\solTY) \right)^2 \!+\! \\
  & 4 \underbrace{\EE_w\left( \normd{A y}^2 (\sigma^2 \dev{\theta}{A}{y} \!-\! \dotp{A w}{A \muTY}) \right)}_{=T}.
\end{align*}
Lemma~\ref{lem:stein} gives $\EE_w\dotp{A w}{A \muTY} = \sigma^2 \EE_w \dev{\theta}{A}{y}$, and we get
\begin{align*}
  T = \;&
  \underbrace{
    2 \EE_w\left( \dotp{A w}{A \mu_0} (\sigma^2 \dev{\theta}{A}{y} - \dotp{A w}{A \muTY}) \right)
  }_{T_1} +\\
  &
  \underbrace{
    \EE_w\left( \normd{A w}^2 (\sigma^2 \dev{\theta}{A}{y} - \dotp{A w}{A \muTY}) \right)
  }_{T_2}.
\end{align*}
Let $\mu^\star_A(y) = A^* A \muTY$, $\mu^0_A = A^* A \mu_0$
and $w_A = A^* A w$. Observe that
$\dev{\theta}{A}{y} = \diverg \mu^\star_A(y)$ and $w_i\pa{\mu^\star_A(y)}_i$ is weakly differentiable. 
Then by integration by parts (in the same vein as in the proof of Stein's Lemma~\ref{lem:stein}), we get
\begin{align*}
  T_1
  &=
  2 \sigma^2 \sum_{i,j} \EE_w \pa{ w_i (\mu^0_A)_i \pd{\mu^\star_A(y)_j}{w_j} }
  -
  2 \sum_{i,j} \EE_w \pa{ w_i (\mu^0_A)_i w_j \mu^\star_A(y)_j }\\
  &=
  -
  2 \sigma^2 \sum_{i,j} \EE_w \pa{ (\mu^0_A)_i \pd{w_i}{w_j} \mu^\star_A(y)_j }
  =
  -2 \sigma^2
  \EE_w \dotp{\mu^0_A}{\mu^\star_A(y)} ~,
\end{align*}
and
\begin{align*}
  T_2
  &=
  \sigma^2 \sum_{i,j} \EE_w \pa{ w_i (w_A)_i \pd{\mu^\star_A(y)_j}{w_j} }
  -
  \sum_{i,j} \EE_w \pa{ w_i (w_A)_i w_j \mu^\star_A(y)_j }\\
  &=
  - \sigma^2 \sum_{i,j} \EE_w \pa{ \pd{w_i (w_A)_i }{w_j} \mu^\star_A(y)_j } \\
  & = - \sigma^2 \sum_{j} \EE_w \pa{ \mu^\star_A(y)_j \left( \sum_i \pd{w_i (w_A)_i }{w_j} \right) }.
\end{align*}
In turn,
\begin{align*}
  \sum_i \pd{w_i (w_A)_i }{w_j}
  = \pd{}{w_j} \normd{A w}^2
  = 2 (A^*A w)_j
  =
  2 (w_A)_j.
\end{align*}
Hence,
\begin{align*}
  T_2
  &=
  - 2 \sigma^2 \sum_{j} \EE_w \pa{ (\mu^\star_A(y))_j (w_A)_j }
  =
  - 2 \sigma^2 \EE_w \dotp{w_A}{\mu^\star_A(y)} \\
  &= 
  - 2 \sigma^2 \EE_w \dotp{A^*A w}{A^*A\mu^\star(y)} = -2 \sigma^4 \EE_w \dev{\theta}{A^*A}{y} ~,
\end{align*}
where the last equality is again a consequence of Lemma~\ref{lem:stein}.
It follows that
$
  T =
  -2 \sigma^2 \left(
  \EE_w \dotp{\mu^0_A}{\mu^\star_A(y)}
  + \sigma^2 \EE_w\dev{\theta}{A^* A}{y}
  \right)
$.
Moreover using \cite[Property~1]{luisier-surelet} we have
\begin{align*}
  &\EE_w \pa{Q^A(\solTY) - \SE^A(\solTY)}^2 \\
  &=
  4 \EE_w \left(
  \sigma^2 \diverg \mu^\star_A(y) - \dotp{w}{\mu_A^\star(y)}
  \right)^2\\
  &=
  4 \sigma^2 \left(
  \EE_w \normd{\mu^\star_A(y)}^2 +
  \sigma^2 \EE_w \tr\bra{ \left(\dfrac{\partial \mu_A^\star(y)}{\partial y} \right)^2 }
  \right).
\end{align*}
Therefore, the reliability is given by
\begin{align*}
  R^A
  =
  2 \sigma^4 \tr[(A^* A)^2]
  +&
  4 \sigma^2 \EE_w \normd{\mu^0_A - \mu^\star_A(y)}^2\\
  &+
  4 \sigma^4
  \EE_w \left( \tr\bra{ \left(\dfrac{\partial \mu_A^\star(y)}{\partial y}\right)^2 }
    - 2 \dev{\theta}{A^* A}{y} \right)~.
\end{align*}
Rearranging the last term above, we obtain the derived expression.
\end{proof}

\subsection{Proof of Corollary \ref{cor:dof}}

\begin{proof}
  Let $\lambda \in \RR_+^*$.
  From Theorem~\ref{thm:dof}(iii), $y \mapsto \Phi \solS$ is differentiable almost everywhere and we can invoke Theorem~\ref{thm:gsure} to derive the GSURE expressions.

  We also observe that $V = \Phi \AJ \Phi^*$ is the orthogonal projector on $\Im V=\Phi(\GJ)$, so that $\tr V = \dim(\Im V)=\rank(\Phi\Proj_{\GJ})$.
  Since $\Phi$ is injective on \GJ under \eqref{eq:hj}, it follows that $\tr V = \dim( \GJ )$.
  Hence, using Theorem~\ref{thm:gsure} with $A=\Id$, Theorem~\ref{thm:dof}(ii) and \eqref{eq:divdof}, it follows that $\dim( \GJ )$ is an unbiased estimator of $\DOF(\lambda)$.
\end{proof}

\subsection{Proof of Corollary \ref{cor:reliability}}
\begin{proof}
  As $V = \Phi \AJ \Phi^*$ is the orthogonal projector on $\Phi(\GJ)$, we have
  \begin{equation*}
    \tr\bra{A V A^* A (2 \Id - V) A^*}
    \geq 0~.
  \end{equation*}
  Moreover $A^* A$ is Hermitian, hence $\tr\bra{(A^*A)^2}=\norm{A^*A}_F^2$, we obtain
  (with the notation of Section \ref{sub:reliability_dof})
  the following upper-bound of the reliability
  \begin{align*}
    R^A \leq
    2 \sigma^4 \norm{A^* A}_F^2
    +
    4 \sigma^2 \norm{A}^4 \EE \normd{\mu_0 - \mus}^2
  \end{align*}
  where $\norm{.}_F$ is the Frobenius norm and
  $\norm{.}$ the matrix spectral norm.
  Then, for $A \in \RR^{M \times Q}$,
  using classical inequalities, we get
  \begin{align*}
    \norm{A^* A}_F^2
    \leq \rank(A) \norm{A}^4 = \min(M, Q) \norm{A}^4 \leq Q \norm{A}^4~.
  \end{align*}
  Since $\solS$ is a solution of \lasso, we have 
  \begin{align*}
  \tfrac{1}{2}\normd{y - \mus}^2 \leq \obj(\solS) \leq \obj(0) = \tfrac{1}{2}\normd{y}^2.
  \end{align*}
  Thus, using Jensen's inequality, we get
  \begin{align*}
    \EE \normd{\mu_0 - \mus}^2
    \leq
    2 \EE &(\normd{\mu_0 - y}^2 + \normd{y - \mus}^2)\\
    &\leq
    2 \EE (\normd{w}^2 + \normd{y}^2)
    \leq 2 \left(
    \normd{\mu_0}^2
    + 2 Q \sigma^2
    \right).
  \end{align*}
  Altogether, this yields the following upper bound
  \begin{align*}
    \frac{R^A}{\sigma^4 Q^2} \leq
    \norm{A}^4
    \left(\frac{18}{Q}
    +
    \frac{8 \normd{\mu_0}^2}{\sigma^2 Q^2}
    \right)~.
  \end{align*}
  Since $\normd{\mu_0} < \infty$,
  this concludes the proof.
\end{proof}

\subsection{Proof of Proposition \ref{prop:gsure_computation}}
\begin{proof}
  We have
  \eq{
    \tr\bra{ A \Phi \AJ \Phi^* A^* } = \tr\bra{ \Phi \AJ \Phi^* A^* A }.
  }
  Hence denoting $\nu(z) = \AJ \Phi^* z$, and using the fact that for any
  matrix $U$,  $\tr U=\EE_Z \dotp{Z}{U Z}$, we arrive at \eqref{eq-prop-calcul-1}.

  We then use the fact that $\AJ\Phi^*$, the inverse of $\Phi$ on $\GJ$, is the mapping that solves the following linearly constrained least-squares problem
  \eq{
    \AJ \Phi^* z =
    \uargmin{h \in \GJ} \normd{\Phi h - z}^2.
  }
  Writing the KKT conditions of this problem leads to \eqref{eq-prop-calcul-2}, where $\tilde \nu$ are the Lagrange multipliers.
\end{proof}